\documentclass[11pt,letterpaper]{amsart}

\usepackage{amsmath,amsthm,amssymb,mathtools,amsfonts,mathrsfs}
\usepackage{xspace,xcolor}
\usepackage[breaklinks,colorlinks,citecolor=teal,linkcolor=teal,urlcolor=teal,pagebackref,hyperindex]{hyperref}
\usepackage[alphabetic]{amsrefs}
\usepackage{tikz-cd}
\usepackage[all]{xy}
\usepackage{bbm}
\usepackage{MnSymbol}
\usepackage{stmaryrd}

\usepackage[bbgreekl]{mathbbol}

\DeclareSymbolFontAlphabet{\mathbb}{AMSb}
\DeclareSymbolFontAlphabet{\mathbbl}{bbold}

\newtheorem{theorem}{Theorem}[section]
\newtheorem{prop}[theorem]{Proposition}
\newtheorem{proposition}[theorem]{Proposition}
\newtheorem{corollary}[theorem]{Corollary}

\newtheorem{conjecture}[theorem]{Conjecture}

\newtheorem{construction}[theorem]{Construction}

\theoremstyle{definition} 
\newtheorem{defn}[theorem]{Definition}
\newtheorem{definition}[theorem]{Definition}
\newtheorem{example}[theorem]{Example}

\newtheorem{remark}[theorem]{Remark}

\newcommand{\R}{\bb{R}}

\newcommand{\qu}{/\kern-.7ex/}
\newcommand{\lqu}{\backslash \kern-.7ex \backslash}%left
\newcommand{\on}{\operatorname} 
\newcommand{\Aut}{\on{Aut}}

\newcommand{\ev}{\on{ev}}

\newcommand{\HH}{\mathfrak H}
\newcommand{\bM}{\overline{\mathcal M}}

\title[LG potential, intrinsic mirror symmetry \& relative mirror map]{The proper Landau--Ginzburg potential, intrinsic mirror symmetry and the relative mirror map}

\author{Fenglong You}
\address{Department of Mathematics \\ ETH Z\"urich, \\Rämistrasse 101, \\8092 Zürich, \\Switzerland}
\email{fenglong.you@math.ethz.ch}

\thanks{}


\begin{document}
\date{\today}

\begin{abstract} 
Given a smooth log Calabi--Yau pair $(X,D)$, we use the intrinsic mirror symmetry construction to define the mirror proper Landau--Ginzburg potential and show that it is a generating function of two-point relative Gromov--Witten invariants of $(X,D)$. We compute certain relative invariants with several negative contact orders, and then apply the relative mirror theorem of \cite{FTY} to compute two-point relative invariants. When $D$ is nef, we compute the proper Landau--Ginzburg potential and show that it is the inverse of the relative mirror map. Specializing to the case of a toric variety $X$, this implies the conjecture of \cite{GRZ} that the proper Landau--Ginzburg potential is the open mirror map. When $X$ is a Fano variety, the proper potential is related to the anti-derivative of the regularized quantum period. 
\end{abstract}

\maketitle 

\tableofcontents

\section{Introduction}

Mirror symmetry is originally stated as a duality between Calabi--Yau manifolds. Mirror symmetry predicts that the symplectic geometry (or the complex geometry) of a Calabi--Yau manifold is equivalent to the complex geometry (or the symplectic geometry, respectively) of the mirror Calabi--Yau manifold. The mirror duality has been generalized to Fano varieties, more generally, to log Calabi--Yau pairs. The mirror of a smooth log Calabi--Yau pair $(X,D)$ is a Landau--Ginzburg model, which is a variety $X^\vee$ with a proper map, called the superpotential, $W:X^\vee \rightarrow \mathbb C$. One further expects that the generic fiber of the superpotential $W$ is mirror to a smooth anticanonical divisor $D$ of $X$. We call the mirror duality between a smooth log Calabi--Yau pair $(X,D)$ and a proper Landau--Ginzburg model \emph{relative mirror symmetry}. To construct the mirror of the smooth log Calabi--Yau pair $(X,D)$, one needs to construct the variety $X^\vee$ and the proper Landau--Ginzburg potential $W$.

The variety $X^\vee$ is considered as the mirror of the complement $X\setminus D$. A general construction of the variety $X^\vee$ is through intrinsic mirror symmetry \cite{GS19} in the  Gross--Siebert program. One considers a maximally unipotent degeneration $g: Y\rightarrow S$, where $S$ is an affine curve, of the pair $(X,D)$. The mirror is constructed as the projective spectrum of the degree zero part of the relative quantum cohomology $QH^0_{\log}(Y,D^\prime)$ of $(Y,D^\prime)$, where $D^\prime$ is a certain divisor that contains $g^{-1}(0)$.

It remains to compute the proper Landau--Ginzburg potential $W$. Following the Gross--Siebert program, the Landau--Ginzburg potentials are given by the theta functions. The theta functions are usually difficult to compute. 

Recently, \cite{GRZ} computed the proper Landau--Ginzburg potentials for toric del Pezzo surfaces. They considered a toric degeneration of the smooth pair $(X,D)$, then applied the tropical view of the Landau--Ginzburg models \cite{CPS}. The theta function in \cite{GRZ} was defined tropically. By proving a tropical correspondence theorem in \cite{Graefnitz2022}, they showed that the theta function can be written as a generating function of two-point relative invariants. The idea of computing two-point relative invariants in \cite{GRZ} was to relate these two-point relative invariants with one-point relative invariants of a blow-up $\tilde {X}$. Then use the local-relative correspondence of \cite{vGGR} to relate these invariants to local invariants of the Calabi--Yau threefold $K_{\tilde {X}}$. By the open-closed duality of \cite{LLW11}, these local invariants are open invariants of the local Calabi--Yau threefold $K_X$ which form the open mirror map. Therefore, \cite{GRZ} showed that, for toric del Pezzo surfaces, the proper Landau--Ginzburg potentials are the open mirror maps.

In this paper, we study the Landau--Ginzburg model from the intrinsic mirror symmetry construction in \cite{GS19}. The goal of this paper is to generalize the result of \cite{GRZ} to all dimensions via a direct computation of two-point relative Gromov--Witten invariants. The computation is based on the relative mirror theorem of \cite{FTY}, where we only need to assume that $D$ is nef. The variety $X$ is not necessarily toric or Fano. 

\subsection{Intrinsic mirror symmetry and theta functions}
Besides the tropical view of the Landau--Ginzburg model \cite{CPS}, the proper Landau--Ginzburg model can also be constructed through intrinsic mirror symmetry. We learnt about the following construction from Mark Gross. 

Given a smooth log Calabi--Yau pair $(X,D)$. We recall the maximally unipotent degeneration $g:Y\rightarrow S$ and the pair $(Y,D^\prime)$ from the construction of $X^\vee$. Following the mirror construction of \cite{GS19}*{Construction 1.19}, the theta functions in $QH^0_{\log}(Y,D^\prime)$ form a graded ring. The degree zero part of the ring agrees with $QH^0_{\log}(X,D)$. The base of the Landau--Ginzburg mirror of $(X,D)$ is $\on{Spec}QH^0_{\log}(X,D)=\mathbb A^1$ and the superpotential is $W=\vartheta_{1}$, the unique primitive theta function of $QH^0_{\log}(X,D)$. 

We claim that the theta functions of $QH^0_{\log}(X,D)$ are generating functions of two-point relative Gromov--Witten invariants as follows.
\begin{definition}[=Definition \ref{def-theta-func}]\label{intro-def-theta}
For $p\geq 1$, the theta function is
\begin{align}\label{intro-theta-func-def}
\vartheta_p=x^{-p}+\sum_{n=1}^{\infty}nN_{n,p}t^{n+p}x^n,
\end{align}
where 
\[
N_{n,p}=\sum_{\beta} \langle [\on{pt}]_n,[1]_p\rangle_{0,2,\beta}^{(X,D)}
\]
is the sum of two-point relative Gromov--Witten invariants with the first marking having contact order $n$ along with a point constraint and the second marking having contact order $p$.
\end{definition}

By \cite{GS19}, theta functions should satisfy the following product rule
\begin{align}\label{intro-theta-func-multi}
\vartheta_{p_1}\star \vartheta_{p_2}=\sum_{r\geq 0, \beta}N_{p_1,p_2,-r}^{\beta} \vartheta_r,
\end{align}
where the structure constants $N_{p_1,p_2,-r}^{\beta}$ are punctured invariants with two positive contacts and one negative contact. 

In Proposition \ref{prop-struc-const}, we show that the structure constants can be written in terms of two-point relative invariants. For example, when $r<p_1,p_2$, we have
\[
N_{p_1,p_2,-r}^{\beta}=(p_1-r)\langle [\on{pt}]_{p_1-r}, [1]_{p_2}\rangle_{0,2,\beta}^{(X,D)}+ (p_2-r)\langle [\on{pt}]_{p_2-r}, [1]_{p_1}\rangle_{0,2,\beta}^{(X,D)}.
\]

In other words, we reduce relative invariants with two positive contacts and one negative contact to relative invariants with two positive contacts. We refer to Proposition \ref{prop-struc-const} for the formula of $N_{p_1,p_2,-r}^{\beta}$ in other cases. 

Then we show that the theta functions in Definition \ref{intro-theta-func-def} indeed satisfy the product rule (\ref{intro-theta-func-multi}) with the correct structure constants $N_{p_1,p_2,-r}^{\beta}$. In particular, in Proposition \ref{prop-wdvv}, we prove an identity of two-point relative invariants generalizing \cite{GRZ}*{Lemma 5.3} and show that it follows from the WDVV equation.

%\textcolor{red}{Maybe write down the statements of the propositions...}

\begin{remark}
During the preparation of our paper, we learnt that Yu Wang\cite{Wang} also obtained the same formula as in Proposition \ref{prop-struc-const} but using the punctured invariants of \cite{ACGS}. Some formulas for two-point relative invariants are also obtained in \cite{Wang} via a different method. 
\end{remark}

\subsection{Relative mirror maps and computing genus zero relative invariants.}

The relative mirror theorem of \cite{FTY} states that, under the assumption that $D$ is nef and $-K_X-D$ is nef, a genus zero generating function of relative Gromov--Witten invariants (the $J$-function) can be identified with the relative periods (the $I$-function) via a change of variables called the relative mirror map. This provides a powerful tool to compute genus zero relative Gromov--Witten invariants.

Our computation of these two-point relative invariants is straightforward but complicated. It is straightforward to see that these invariants can be extracted from the relative $J$-function after taking derivatives. On the other hand, although such computation for (one-point) absolute invariants is well-known, the computation of two-point relative invariants is much more complicated due to the following reasons. 

First of all, we need to compute two-point relative invariants instead of one-point invariants. For one-point relative invariants, one can also use the local-relative correspondence of \cite{vGGR} (see also \cite{TY20b}) to reduce the computation to local invariants when the divisor $D$ is nef. To compute these two-point invariants one need to consider the so-called extended relative $I$-function, instead of the much easier non-extended relative $I$-function. 

Secondly, the relative mirror map has never been studied systematically. There have been some explicit computations of relative invariants when the relative mirror maps are trivial, see, for example, \cite{TY20b}. When the relative mirror maps are not trivial, more complicated invariants will appear. One of the important consequences of this paper is to provide a systematic analysis of these invariants and set up the foundation of future applications of the relative mirror theorem.

We would like to point out a related computation in \cite{You20}, where we computed one-point relative invariants of some partial compactifications of toric Calabi--Yau orbifolds. The computation is much easier in \cite{You20} because of the two reasons that we just mentioned. First of all, one can apply the local-relative correspondence to compute these invariants, although we did not use it in \cite{You20}. Secondly, although the mirror map is not trivial, the mirror map is essentially coming from the absolute Gromov--Witten theory of the partial compactifications. The relative theory in \cite{You20} does not contribute to the non-trivial mirror map. Therefore we were able to avoid all these complexities. We would also like to point out that the computation in \cite{You20} is restricted to the toric case. In this paper, we work beyond the toric setting.

In order to apply the relative mirror theorem of \cite{FTY}, we need to study the relative mirror map carefully.
Let $X$ be a smooth projective variety and $D$ be a smooth nef divisor, we recall that the $J$-function for the pair $(X,D)$ is defined as
\[
J_{(X,D)}(\tau,z)=z+\tau+\sum_{\substack{(\beta,l)\neq (0,0), (0,1)\\ \beta\in \on{NE(X)}}}\sum_{\alpha}\frac{q^{\beta}}{l!}\left\langle \frac{\phi_\alpha}{z-\bar{\psi}},\tau,\ldots, \tau\right\rangle_{0,1+l, \beta}^{(X,D)}\phi^{\alpha},
\]
and
the (non-extended) $I$-function of the smooth pair $(X,D)$ is 
\[
I_{(X,D)}(y,z)=\sum_{\beta\in \on{NE(X)}}J_{X,\beta}(\tau_{0,2},z)y^\beta\left(\prod_{0<a\leq D\cdot \beta-1}(D+az)\right)[1]_{-D\cdot \beta}.
\]
We refer to Definition \ref{def-relative-J-function} and Definition \ref{def-relative-I-function} for the precise meaning of the notation. The extended $I$-function $I(y,x_1,z)$ takes a more complicated form than the non-extended $I$-fnction $I(y,z)$. We refer to Definition \ref{def-relative-I-function-extended} for the precise definition of the extended $I$-function. We further assume that $-K_X-D$ is nef. The extended relative mirror map is given by the $z^0$-coefficient of the extended $I$-function:
\begin{align*}
\tau(y,x_1)=\sum_{i=1}^r p_i\log y_i+x_1[1]_{1}+\sum_{\substack{\beta\in \on{NE}(X)\\ D\cdot \beta \geq 2}}\langle [\on{pt}]\psi^{D\cdot \beta-2}\rangle_{0,1,\beta}^Xy^\beta (D\cdot \beta-1)![1]_{-D\cdot \beta}.
\end{align*}
The (non-extended) relative mirror map is given by $\tau(y,0)$, denoted by $\tau(y)$. The relative mirror theorem of \cite{FTY} states that
\[
J(\tau(y,x_1),z)=I(y,x_1,z).
\]

Therefore from the expression of $\tau(y,x_1)$, we can see that relative invariants with several negative contact orders will appear when the relative mirror map is not trivial. We obtain the following identity which shows that the negative contact insertion $[1]_{-k}$ is similar to the insertion of a divisor class $[D]_0$. 
\begin{proposition}[=Proposition \ref{prop-several-neg-1}]
\begin{align*}
\langle [1]_{-k_1},\cdots, [1]_{-k_l}, [\gamma]_{k_{l+1}} \bar{\psi}^a\rangle_{0,l+1,\beta}^{(X,D)}=\langle [D]_0,\cdots, [D]_0, [\gamma]_{D\cdot \beta} \bar{\psi}^a\rangle_{0,l+1,\beta}^{(X,D)},
\end{align*}
where $\gamma \in H^*(D)$, $k_i$'s are positive integers, and 
\[
D\cdot \beta=k_{l+1}-\sum_{i=1}^l k_i\geq 0.
\]
\end{proposition}

With these preparations, we are able to express invariants in $J(\tau(y,0),z)$ in terms of relative invariants without negative contact orders and the relative mirror map can be written as the following change of variables
\begin{align}\label{intro-relative-mirror-map}
\sum_{i=1}^r p_i\log q_i=\sum_{i=1}^r p_i\log y_i+g(y)D.
\end{align}

Since we need to compute invariants with an insertion $[1]_1$, we also need to compute the following relative invariants with insertions of $[1]_1$. 

\begin{itemize}
    \item Relative invariants with two positive contact orders and several negative contact orders of the following form:
    \[
    \langle [1]_1, [1]_{-k_1},\cdots, [1]_{-k_l}, [\gamma]_{k_{l+1}} \rangle_{0,l+2,\beta}^{(X,D)},
    \]
     \[
    D\cdot \beta-1=k_{l+1}-\sum_{i=1}^l k_i\geq 0, \text{ and } k_i>0.
    \]
    In Proposition \ref{prop-several-neg-2}, we show that these invariants equal to two point invariants
    \[
    (D\cdot \beta-1)^l\langle [1]_1, [\gamma]_{D\cdot \beta-1} \rangle_{0,2,\beta}^{(X,D)}.
    \] 
    For invariants with several insertions of $[1]_1$, we refer to Proposition \ref{prop-several-neg-3}.
    \item Degree zero relative invariants  with two positive contact orders and several negative contact orders of the following form:
    \[
    \langle [1]_1,[1]_{-k_1},\cdots, [1]_{-k_l},[\on{pt}]_{k_{l+1}}\rangle_{0,l+2,0}^{(X,D)}.
    \]
    By Proposition \ref{prop-degree-zero}, these invariants equal to $(-1)^{l-1}$.
\end{itemize}

\begin{remark}
We would like to point out that a key point of the computation of the proper potential is to observe the subtle (but vital) difference between the above mentioned formulas for invariants with two positive contacts and formulas for invariants with one positive contact. 
\end{remark}

\begin{remark}
The proof of Proposition \ref{prop-several-neg-1}, Proposition \ref{prop-several-neg-2}, Proposition \ref{prop-several-neg-3}, and Proposition \ref{prop-degree-zero} make essential use of the (both orbifold and graph sum) definitions of relative Gromov--Witten invariants with negative contact orders in \cite{FWY}. Since these invariants reduce to relative invariants with two positive contact orders, we do not need to assume a general relation between the punctured invariants of \cite{ACGS} and \cite{FWY}. To match with the intrinsic mirror symmetry in the Gross--Siebert program, we only need to assume that punctured invariants of a smooth pair with one negative contact order defined in \cite{ACGS} coincide with the ones in \cite{FWY}. A more general comparison result is an upcoming work of \cite{BNR22}, so we do not attempt to give a proof of this special case.
\end{remark}

Relative mirror theorem was established in \cite{FTY}, however, it was not clear how to apply the relative mirror theorem to compute relative invariants when the mirror map is not trivial. Understanding these invariants with several negative contact orders are essential for the application of the relative mirror theorem to compute relative invariants when the mirror map is not trivial. Our approach provides a powerful tool to compute genus zero relative invariants. It is not limited to two-point invariants, it can be generalized to compute genus zero relative invariants with several relative markings with insertions $k_i\geq 1$.

\subsection{The proper Landau--Ginzburg potential}

The main result of the paper is to relate the proper Landau--Ginzburg potential with the relative mirror map. Note that the generating function $\vartheta_1$ in Definition \ref{intro-def-theta} is not simply a coefficient of the $J$-function, instead, it is a derivative of a coefficient of the $J$-function. By extracting a coefficient of the $J$-function and the $I$-function, applying the computation of relative invariants with several negative contact orders in Proposition \ref{prop-several-neg-2} and Proposition \ref{prop-degree-zero}, and taking derivatives, we have the following statement.

\begin{theorem}[=Theorem \ref{thm-main}]\label{intro-thm-main}
Let $X$ be a smooth projective variety with a smooth nef anticanonical divisor $D$. Let $W:=\vartheta_1$ be the mirror proper Landau--Ginzburg potential. Set $q^\beta=t^{D\cdot \beta}x^{D\cdot\beta}$. Then
\[
W=x^{-1}\exp\left(g(y(q))\right),
\]
where
\[
g(y)=\sum_{\substack{\beta\in \on{NE}(X)\\ D\cdot \beta \geq 2}}\langle [\on{pt}]\psi^{D\cdot \beta-2}\rangle_{0,1,\beta}^Xy^\beta (D\cdot \beta-1)!
\]
and $y=y(q)$ is the inverse of the relative mirror map (\ref{intro-relative-mirror-map}).
\end{theorem}

\begin{remark}
This is a natural expectation from the point of view of relative mirror symmetry. Recall that the proper Landau--Ginzburg model $(X^\vee,W)$ is mirror to the smooth log Calabi--Yau pair $(X,D)$. The proper Landau--Ginzburg potential $W$ should encode the instanton corrections. On the other hand, the relative mirror theorem relates relative Gromov--Witten invariants with relative periods (relative $I$-functions) via the relative mirror map. In order to have a mirror construction with a trivial mirror map, the instanton corrections should be the inverse relative mirror map. This provides an enumerative meaning of the relative mirror map.
\end{remark}

\begin{remark}
In \cite{CJL21} and \cite{CJL20}, Collins--Jacob--Lin constructed the SYZ fibration and the dual fibration with respect to the Ricci-flat metric for del Pezzo surfaces and rational elliptic surfaces. It would be interesting to see how to extract quantum corrections from their construction, then compare with the computation from the Gross--Siebert program. There are some related computation in \cite{Lin20}, \cite{LLL} and \cite{BECHL}.
\end{remark}

\begin{remark}
    Applications of mirror theorems (e.g. \cite{CCLT}, \cite{CLT}, \cite{CLLT} and \cite{You20}) are usually for toric targets. Theorem \ref{intro-thm-main} holds for general targets based on highly non-trivial computations. 
\end{remark}

In \cite{GRZ}, the authors conjectured that the proper Landau--Ginzburg potential is the open mirror map. We also have a natural explanation for it. The open mirror map of \cite{GRZ} is given by open Gromov--Witten invariants of the local Calabi--Yau $\mathcal O_X(-D)$. These open invariants encode the instanton corrections and are expected to be the inverse mirror map of the local Gromov--Witten theory of $\mathcal O_X(-D)$. We observe that the relative mirror map (\ref{intro-relative-mirror-map}) and the local mirror map coincide up to a sign. So we claim that the proper Landau--Ginzburg potential is also the open mirror map. When $X$ is a toric variety, we proved the conjecture of \cite{GRZ}.
\begin{theorem}[=Theorem \ref{thm-toric-open}]\label{intro-thm-toric-open}
Let $(X,D)$ be a smooth log Calabi--Yau pair, such that $X$ is toric and $D$ is nef. The proper Landau--Ginzburg potential of $(X,D)$ is the open mirror map of the local Calabi--Yau manifold $\mathcal O_X(-D)$.
\end{theorem}

In general, Theorem \ref{intro-thm-toric-open} is true as long as the open-closed duality (e.g. \cite{CLT} and \cite{CCLT}) between open Gromov--Witten invariants $K_X$ and closed Gromov--Witten invariants $P(\mathcal O(-D)\oplus \mathcal O)$ is true. Therefore, we have the following.

\begin{corollary}
The open-closed duality implies the proper Landau--Ginzburg potential is the open mirror map.
\end{corollary}

\begin{remark}
The reason that the relative mirror map is the same as the local mirror map can also be seen from the local-relative correspondence of \cite{vGGR} and \cite{TY20b}. It has already been observed in \cite{TY20b} that the local and (non-extended) relative $I$-functions can be identified. And this identification has been used to prove the local-relative correspondence for some invariants in \cite{TY20b}.  
\end{remark}

\begin{remark}
We consider Theorem \ref{intro-thm-main} provides a complete story from algebro-geometric point of view. It provides a connection between the Gross--Siebert mirror construction \cite{GS19} and the relative version of the enumerative mirror symmetry in \cite{FTY}. For the SYZ mirror symmetry, one may naturally expect that the proper potential is a generating function of genus zero open Gromov--Witten invariants of $X\setminus D$ \cite{Auroux07}. Open Gromov--Witten invariants of $\mathcal O_X(-D)$ appear in Theorem \ref{intro-thm-toric-open} maybe because there exists an identity between open Gromov--Witten invariants of $\mathcal O_X(-D)$ and the open Gromov--Witten invariants of $X\setminus D$, similar to the local-relative correspondence \cite{vGGR} for closed Gromov--Witten invariants. However, we do not know if such an identity will be true in general. We refer to \cite{Lin20} for some discussions about open Gromov--Witten invariants of $X\setminus D$.
\end{remark}

Theorem \ref{intro-thm-main} provides explicit formulas for the proper potentials whenever the relevant genus zero absolute Gromov--Witten invariants of $X$ are computable. These absolute invariants can be extracted from the $J$-function of the absolute Gromov--Witten theory of $X$. Therefore, we have explicit formulas of the proper Landau--Ginzburg potentials whenever a Givental style mirror theorem holds. Givental style mirror theorem has been proved for many cases beyond the toric setting (e.g. \cite{CFKS} for non-abelian quotients via the abelian/non-abelian correspondence). Therefore, we have explicit formulas for the proper Landau--Ginzburg potentials for large classes of examples. Note that there may be non-trivial mirror maps for absolute Gromov--Witten theory of $X$. If we replace the absolute invariants in $g(y)$ by the corresponding coefficients of the absolute $I$-function, we also need to plug-in the inverse of the absolute mirror map. This can be seen in the case of toric varieties in Section \ref{sec-toric-semi-Fano}.

For Fano varieties, the invariants in $g(y)$ are usually easier. We observed that $g(y)$ is closely related to the regularized quantum periods in the Fano search program \cite{CCGGK}.

\begin{theorem}\label{intro-thm-quantum-period}
The function $g(y)$ coincides with the anti-derivative of the regularized quantum period.
\end{theorem}

By mirror symmetry, it is expected that regularized quantum periods of Fano varieties coincide with the classical periods of their mirror Laurent polynomials. Therefore, as long as one knows the mirror Laurent polynomials, one can compute the proper Landau--Ginzburg potentials. For example, the proper Landau--Ginzburg potentials for all Fano threefolds can be explicitly computed using \cite{CCGK}. More generally, Theorem \ref{intro-thm-quantum-period} allows one to use the large databases \cite{CK22} of quantum periods for Fano manifolds to compute the proper Landau--Ginzburg potentials.

Interestingly, the Laurent polynomials are considered as the mirror of Fano varieties with maximal boundaries (or as the potential for the weak, non-proper, Landau--Ginzburg models of \cite{Prz07}, \cite{Prz13}). Therefore, we have an explicit relation between the proper and non-proper Landau--Ginzburg potentials.

\subsection{Acknowledgement}
The author would like to thank Mark Gross for explaining the construction of Landau--Ginzburg potentials from intrinsic mirror symmetry. The author would also like to thank Yu Wang for illuminating discussions regarding Section \ref{sec:intrinsic}. This project has received funding from the Research Council of Norway grant no. 202277 and the European Union’s Horizon 2020 research and innovation programme under the Marie Skłodowska -Curie grant agreement 101025386.

\section{Relative Gromov--Witten invariants with negative contact orders}

\subsection{General theory}\label{sec-rel-general}

%\textcolor{red}{ Revise this section to avoid overlap}

We follow the presentation of \cite{FWY} for the definition of genus zero relative Gromov--Witten theory with negative contact orders.

Let $X$ be a smooth projective variety and $D\subset X$ be a smooth divisor. We consider a topological type (also called an admissible graph)
\[
\Gamma=(0,n,\beta,\rho,\vec \mu)
\]
with 
\[
\vec \mu=(\mu_1, \ldots, \mu_\rho)\in (\mathbb Z^*)^{\rho}
\]
and 
\[
\sum_{i=1}^\rho \mu_i=\int_\beta D.
\]

\begin{defn}[\cite{FWY}, Definition 2.4]
A rubber graph $\Gamma'$ is an admissible graph whose roots have two different types. There are
\begin{enumerate}
    \item $0$-roots (whose weights will be denoted by $\mu^0_1,\ldots,\mu^0_{\rho_0}$), and
    \item $\infty$-roots (whose weights will be denoted by $\mu^\infty_1,\ldots,\mu^\infty_{\rho_\infty}$).
\end{enumerate}
The map $b$ maps $V(\Gamma)$ to $H_2(D,\mathbb Z)$.
\end{defn}

\begin{defn}[\cite{FWY}, Definition 4.1]\label{def:admgraph0}
\emph{A (connected) graph of type $0$} is a weighted graph $\Gamma^0$ consisting of a
single vertex, no edges, and the following.
\begin{enumerate}
\item {$0$-roots},
\item {$\infty$-roots of node type},
\item {$\infty$-roots of marking type},
\item {Legs}.
\end{enumerate}
$0$-roots are weighted by
positive integers, and $\infty$-roots are weighted by negative integers. The vertex is associated with a tuple $(g,\beta)$ where $g\geq 0$
and $\beta\in H_2(D,\mathbb Z)$. 
\end{defn}

A graph $\Gamma^\infty$ of type $\infty$ is an admissible graph such that the roots are distinguished by node type and marking type.

\begin{defn}[\cite{FWY}, Definition 4.8]\label{defn:locgraph}
\emph{An admissible bipartite graph} $\mathfrak G$ is a tuple
$(\mathfrak S_0,\Gamma^\infty,I,E,g,b)$, where 
\begin{enumerate}
\item $\mathfrak S_0=\{\Gamma_i^0\}$ is a set of graphs of type $0$;
  $\Gamma^\infty$ is a (possibly disconnected) graph of type $\infty$.
\item $E$ is a set of edges.
\item $I$ is the set of markings.
\item $g$ and $b$ represent the genus and the degree respectively.
\end{enumerate}
Moreover, the admissible bipartite graph must satisfy some conditions described in \cite{FWY}*{Definition 4.8}. We refer to \cite{FWY} for more details.
\end{defn}

Let $\mathcal B_\Gamma$ be a connected admissible bipartite graph of topological type $\Gamma$. Given a bipartite graph $\mathfrak G\in \mathcal B_\Gamma$, we consider
\[
\bM_{\mathfrak G}=\prod_{\Gamma_i^0\in \mathfrak S_0}\bM_{\Gamma_i^0}^\sim (D) \times_{D^{|E|}}\bM_{\Gamma^\infty}^{\bullet}(X,D),
\]
where
\begin{itemize}
    \item $\bM_{\Gamma_i^0}^\sim (D)$ is the moduli space of relative stable maps to rubber target over $D$ of type $\Gamma_i^0$;
    \item $\bM_{\Gamma^\infty}^{\bullet}(X,D)$ is the moduli space of relative stable maps of type $\Gamma^\infty$;
    \item $\times_{D^{|E|}}$ is the fiber product identifying evaluation maps according to edges. 
\end{itemize}
We have the following diagram.
\begin{equation}\label{eqn:diag}
\xymatrix{
\bM_{\mathfrak G} \ar[r]^{} \ar[d]^{\iota} & D^{|E|} \ar[d]^{\Delta} \\
\prod\limits_{\Gamma^0_i\in \mathfrak S_0}\bM^\sim_{\Gamma^0_i}(D) \times \bM^{\bullet}_{\Gamma^\infty}(X,D) \ar[r]^{} & D^{|E|}\times D^{|E|}.
}
\end{equation}
There is a natural virtual class
\[
[\bM_{\mathfrak G}]^{\on{vir}}=\Delta^![\prod\limits_{\Gamma^0_i\in \mathfrak S_0}\bM^\sim_{\Gamma^0_i}(D) \times \bM^{\bullet}_{\Gamma^\infty}(X,D)]^{\on{vir}}
\]
where $\Delta^!$ is the Gysin map.

For each $\bM^\sim_{\Gamma_i^0}(D)$, we have a stabilization map $\bM^\sim_{\Gamma_i^0}(D) \rightarrow \bM_{0,n_i+\rho_i}(D,\beta_i)$ where $n_i$ is the number of legs, $\rho_i$ is the number of $0$-roots plus the number of $\infty$-roots of marking type, and $\beta_i$ is the curve class of $\Gamma_i^0$. Hence, we have a map 
\[
\bM_{\mathfrak G}=\prod\limits_{\Gamma^0_i\in \mathfrak S_0}\bM^\sim_{\Gamma^0_i}(D) \times_{D^{|E|}} \bM^{\bullet}_{\Gamma^\infty}(X,D)
\rightarrow 
\prod\limits_{\Gamma^0_i\in \mathfrak S_0}\bM_{0,n_i+\rho_i}(D,\beta_i) \times_{D^{|E|}} \bM^{\bullet}_{\Gamma^\infty}(X,D).
\]

Composing with the boundary map
\[
\prod\limits_{\Gamma^0_i\in \mathfrak S_0}\bM_{0,n_i+\rho_i}(D,\beta_i) \times_{D^{|E|}} \bM^{\bullet}_{\Gamma^\infty}(X,D) \rightarrow \bM_{0,n+\rho}(X,\beta)\times_{X^{\rho}} D^{\rho},
\]
 we obtain a map 
\[
\mathfrak t_{\mathfrak G}:\bM_{\mathfrak G}\rightarrow \bM_{0,n+\rho}(X,\beta)\times_{X^{\rho}} D^{\rho}.
\]

Following the definition in \cite{FWY}, we need to introduce the relative Gromov--Witten cycle of the pair $(X,D)$ of topological type $\Gamma$.

Let $t$ be a formal parameter. Given a $\Gamma^\infty$, we define
\begin{align}\label{neg-rel-infty}
C_{\Gamma^\infty}(t)=\frac{t}{t+\Psi}\in A^*(\bM^\bullet_{\Gamma^\infty}(X,D))[t^{-1}].
\end{align}

Then we consider $\Gamma_i^0$. Define
\[
c(l)=\Psi_\infty^l-\Psi_\infty^{l-1}\sigma_1+\ldots+(-1)^l\sigma_l,
\]
where $\Psi_\infty$ is the divisor corresponding to the cotangent line bundle determined by the relative divisor on the $\infty$ side. We then define
\[
\sigma_k=\sum\limits_{\{e_1,\ldots,e_k\}\subset \on{HE}_{m,n}(\Gamma_i^0)} \prod\limits_{j=1}^{k} (d_{e_j}\bar\psi_{e_j}-\ev_{e_j}^*D),
\]
where $d_{e_j}$ is the absolute value of the weight at the root $e_j$. 

For each $\Gamma_i^0$, define
\begin{align}\label{neg-rel-0}
C_{\Gamma_i^0}(t)= \frac{\sum_{l\geq 0}c(l) t^{\rho_\infty(i)-1-l}}{\prod\limits_{e\in \on{HE}_{n}(\Gamma_i^0)} \big(\frac{t+\ev_e^*D}{d_e}-\bar\psi_e\big) } \in A^*(\bM^\sim_{\Gamma_i^0}(D))[t,t^{-1}],
\end{align}
where $\rho_\infty(i)$ is the number of $\infty$-roots (of both types) associated with $\Gamma_i^0$.

For each $\mathfrak G$, we write
\begin{equation}\label{eqn:cg}
C_{\mathfrak G}=\left[ p_{\Gamma^\infty}^*C_{\Gamma^\infty}(t)\prod\limits_{\Gamma_i^0\in \mathfrak S_0} p_{\Gamma_i^0}^*C_{\Gamma_i^0}(t) \right]_{t^{0}}, 
\end{equation}
where 
$[\cdot]_{t^{0}}$ means taking 
the constant term, and $p_{\Gamma^\infty}, p_{\Gamma_i^0}$ are projections from $\prod\limits_{\Gamma^0_i\in \mathfrak S_0}\bM^\sim_{\Gamma^0_i}(D) \times \bM^{\bullet}_{\Gamma^\infty}(X,D)$ to the corresponding factors. Recall that
\[
\iota: \bM_{\mathfrak G}\rightarrow \prod\limits_{\Gamma^0_i\in \mathfrak S_0}\bM^\sim_{\Gamma^0_i}(D) \times \bM^{\bullet}_{\Gamma^\infty}(X,D)
\]
is the closed immersion from diagram \eqref{eqn:diag}.

\begin{defn}[\cite{FWY}*{Definition 5.3}]\label{def-rel-cycle}
The relative Gromov--Witten cycle of the pair $(X,D)$ of topological type $\Gamma$ is defined to be 
\[
\mathfrak c_\Gamma(X/D) = \sum\limits_{\mathfrak G \in \mathcal B_\Gamma} \dfrac{1}{|\Aut(\mathfrak G)|}(\mathfrak t_{\mathfrak G})_* ({\iota}^* C_{\mathfrak G} \cap [\bM_{\mathfrak G}]^{\on{vir}}) \in A_*(\bM_{0,n+\rho}(X,\beta)\times_{X^{\rho}} D^{\rho}),
\]
where $\iota$ is the vertical arrow in Diagram \eqref{eqn:diag}.
\end{defn}

\begin{proposition}[=\cite{FWY}, Proposition 3.4]
\[\mathfrak c_\Gamma(X/D) \in A_{d}(\overline{M}_{0,n+\rho}(X,\beta)\times_{X^{\rho}} D^{\rho}),\] where 
\[d=\mathrm{dim}_{\mathbb C}X-3+\int_{\beta} c_1(T_X(-\mathrm{log} D)) + n + \rho_+, \]
where $\rho_+$ is the number of relative markings with positive contact.
\end{proposition}

Let 
\begin{align*}
   \alpha_i\in H^*(X), \text{ and } a_i\in \mathbb Z_{\geq 0}  \text{ for } i\in \{1,\ldots, n\}; 
\end{align*}
\begin{align*}
     \epsilon_j\in H^*(D), \text{ and } b_j\in \mathbb Z_{\geq 0}  \text{ for } j\in \{1,\ldots, \rho\}.
\end{align*}
We have evaluation maps
\begin{align*}
\ev_{X,i}:\bM_{0,n+\rho}(X,\beta)\times_{X^{\rho}} D^{\rho}&\rightarrow X, \text{ for } i\in\{1,\ldots,n\};\\
\ev_{D,j}:\bM_{0,n+\rho}(X,\beta)\times_{X^{\rho}} D^{\rho}&\rightarrow D \text{ for } j\in \{1,\ldots, \rho\}.
\end{align*}

\begin{defn}[\cite{FWY}, Definition 5.7]\label{rel-inv-neg}
The relative Gromov--Witten invariant of topological type $\Gamma$ is
\[
\langle \prod_{i=1}^n \tau_{a_i}(\alpha_i) \mid \prod_{j=1}^\rho\tau_{b_j}(\epsilon_j) \rangle_{\Gamma}^{(X,D)} =  \displaystyle\int_{\mathfrak c_\Gamma(X/D)} \prod\limits_{j=1}^{\rho} \bar{\psi}_{D,j}^{b_j}\ev_{D,j}^*\epsilon_j\prod\limits_{i=1}^n \bar{\psi}_{X,i}^{a_i}\ev_{X,i}^*\alpha_i,
\]
where $\bar\psi_{D,j}, \bar\psi_{X,i}$ are pullback of psi-classes from $\bM_{0,n+\rho}(X,\beta)$ to $\bM_{0,n+\rho}(X,\beta)\times_{X^{\rho}} D^{\rho}$ corresponding to markings.
\end{defn}

\begin{remark}
In \cite{FWY}, relative Gromov--Witten invariants of $(X,D)$ with negative contact order are also defined as a limit of the corresponding orbifold Gromov--Witten invariants of the $r$-th root stack $X_{D,r}$ with large ages
\[
\langle \prod_{i=1}^{n+\rho} \tau_{a_i}(\gamma_i)\rangle_{0,n+\rho,\beta}^{(X,D)}:=r^{\rho_-}\langle \prod_{i=1}^{n+\rho} \tau_{a_i}(\gamma_i)\rangle_{0,n+\rho,\beta}^{X_{D,r}},
\]
where $r$ is sufficiently large; $\rho_-$ is the number of orbifold markings with large ages (=relative markings with negative contact orders); $\gamma_i$ are cohomology classes of $X$ or $D$ depending on the markings being interior or orbifold/relative.
We refer to \cite{FWY}*{Section 3} for more details. 
\end{remark}

We recall the topological recursion relation and the WDVV equation which will be used later in our paper.

\begin{prop}[\cite{FWY}*{Proposition 7.4}]\label{prop:TRR}
Relative Gromov--Witten theory satisfies the topological recursion relation:
\begin{align*}
    &\langle\bar\psi^{a_1+1}[\alpha_1]_{i_1}, \ldots, \bar\psi^{a_n}[\alpha_n]_{i_n}\rangle_{0,\beta,n}^{(X,D)} \\
    =&\sum \langle\bar\psi^{a_1}[\alpha_1]_{i_1}, \prod\limits_{j\in S_1} \bar\psi^{a_j}[\alpha_j]_{i_j}, \widetilde T_{i,k}\rangle_{0,\beta_1,1+|S_1|}^{(X,D)} \langle\widetilde T_{-i}^k, \bar\psi^{a_2}[\alpha_2]_{i_2}, \bar\psi^{a_3}[\alpha_3]_{i_3}, \prod\limits_{j\in S_2} \bar\psi^{a_j}[\alpha_j]_{i_j}\rangle_{0,\beta_2,2+|S_2|}^{(X,D)},
\end{align*}
where the sum is over all $\beta_1+\beta_2=\beta$, all indices $i,k$ of basis, and $S_1, S_2$ disjoint sets with $S_1\cup S_2=\{4,\ldots,n\}$.
\end{prop}
\begin{prop}[\cite{FWY}*{Proposition 7.5}]\label{prop:WDVV}
Relative Gromov--Witten theory satisfies the WDVV equation:
\begin{align*}
    &\sum \langle\bar\psi^{a_1}[\alpha_1]_{i_1}, \bar\psi^{a_2}[\alpha_2]_{i_2}, \prod\limits_{j\in S_1} \bar\psi^{a_j}[\alpha_j]_{i_j}, \widetilde T_{i,k}\rangle_{0,\beta_1,2+|S_1|}^{(X,D)} \\
    & \quad \cdot \langle\widetilde T_{-i}^k, \bar\psi^{a_3}[\alpha_3]_{i_3}, \bar\psi^{a_4}[\alpha_4]_{i_4}, \prod\limits_{j\in S_2} \bar\psi^{a_j}[\alpha_j]_{i_j}\rangle_{0,\beta_2,2+|S_2|}^{(X,D)} \\
    =&\sum \langle\bar\psi^{a_1}[\alpha_1]_{i_1}, \bar\psi^{a_3}[\alpha_3]_{i_3}, \prod\limits_{j\in S_1} \bar\psi^{a_j}[\alpha_j]_{i_j}, \widetilde T_{i,k}\rangle_{0,\beta_1,2+|S_1|}^{(X,D)} \\
    & \quad \cdot \langle\widetilde T_{-i}^k, \bar\psi^{a_2}[\alpha_2]_{i_2}, \bar\psi^{a_4}[\alpha_4]_{i_4}, \prod\limits_{j\in S_2} \bar\psi^{a_j}[\alpha_j]_{i_j}\rangle_{0,\beta_2,2+|S_2|}^{(X,D)},
\end{align*}
where each sum is over all $\beta_1+\beta_2=\beta$, all indices $i,k$ of basis, and $S_1, S_2$ disjoint sets with $S_1\cup S_2=\{5,\ldots,n\}$. 
\end{prop}

\subsection{A special case}

As explained in \cite{FWY}*{Example 5.5}, relative Gromov--Witten invariants with one negative contact order can be written down in a simpler form. In this case, we only have the graphs $\mathfrak G$ such that $\{\Gamma_i^0\}$ consists of only one element (denoted by $\Gamma^0$). Denote such a set of graphs by $\mathcal B_\Gamma'$. The relative Gromov--Witten cycle of topological type $\Gamma$ is simply
\[
\mathfrak c_\Gamma(X/D) = \sum\limits_{\mathfrak G \in \mathcal B_\Gamma'} \dfrac{\prod_{e\in \on{HE}_{n}(\Gamma^0)}d_e}{|\Aut(\mathfrak G)|}(\mathfrak t_{\mathfrak G})_* \big([\overline{\mathcal M}_{\mathfrak G}]^{\on{vir}}\big).
\]
Note that
\[
[\overline{\mathcal M}_{\mathfrak G}]^{\on{vir}}=\Delta^![\overline{\mathcal M}^\sim_{\Gamma^0}(D)\times \overline{\mathcal M}^\bullet_{\Gamma^\infty}(X,D)]^{\on{vir}}.
\]

Let 
\begin{align*}
   \alpha_i\in H^*(X), \text{ and } a_i\in \mathbb Z_{\geq 0}  \text{ for } i\in \{1,\ldots, n\}; 
\end{align*}
\begin{align*}
     \epsilon_j\in H^*(D), \text{ and } b_j\in \mathbb Z_{\geq 0}  \text{ for } j\in \{1,\ldots, \rho\}.
\end{align*}

Without loss of generality, we assume that $\epsilon_1$ is the insertion that corresponds to the unique negative contact marking. Then the relative invariant with one negative contact order can be written as
\begin{align*}
    &\langle \prod_{i=1}^n \tau_{a_i}(\alpha_i) \mid \prod_{j=1}^\rho\tau_{b_j}(\epsilon_j) \rangle_{\Gamma}^{(X,D)}\\
    =& \sum_{\mathfrak G\in \mathcal B_\Gamma'} \frac{\prod_{e\in E}d_e}{|\Aut(E)|}\sum \langle \prod_{j\in S_{\epsilon,1}}\tau_{b_j}(\epsilon_j) | \prod_{i\in S_{\alpha,1}}\tau_{a_i}(\alpha_i) | \eta, \tau_{b_1}(\epsilon_1)\rangle^\sim_{\Gamma^0}\langle \check{\eta}, \prod_{j\in S_{\epsilon,2}}\tau_{b_j}(\epsilon_j) | \prod_{i\in S_{\alpha,2}}\tau_{a_i}(\alpha_i) \rangle_{\Gamma^\infty}^{\bullet, (X,D)},
\end{align*}
where $\Aut(E)$ is the permutation group of the set $\{d_1,\ldots, d_{|E|}\}$; $\check{\eta}$ is defined by taking the Poincar\'e dual of the cohomology weights of the cohomology weighted partition $\eta$;  the second sum is over all splittings of 
\[
\{1,\ldots, n\}=S_{\alpha,1}\sqcup S_{\alpha,2}, \quad \{2,\ldots, \rho\}=S_{\epsilon,1}\sqcup S_{\epsilon,2}
\]
and all intermediate cohomology weighted partitions $\eta$.

The following comparison theorem between punctured invariants of \cite{ACGS} and relative invariants with one negative contact order of \cite{FWY} for smooth pairs is an upcoming work of \cite{BNR22b}.

\begin{theorem}\label{theorem-puncture-relative}
Given a smooth projective variety $X$ and a smooth divisor $D\subset X$, the punctured Gromov--Witten invariants of $(X,D)$ and the relative Gromov--Witten invariants of $(X,D)$ with one negative contact order coincide.
\end{theorem}

\begin{remark}
\cite{BNR22b} studies the comparison between punctured invariants of \cite{ACGS} and relative invariants with several negative contact orders. For the purpose of this paper, we only need the case with one negative contact order. In this case, the comparison is significantly simpler because we have simple graphs as described above and the class $\mathfrak c_\Gamma(X/D)$ is trivial. Since the general comparison is obtained in \cite{BNR22b}, we do not attempt to give a proof for this special case. 

Theorem \ref{theorem-puncture-relative} is sufficient for us to fit our result into the Gross--Siebert program as theta functions and structure constants in \cite{GS19} and \cite{GS21} only involve punctured invariants with one punctured marking. Note that relative invariants with several negative contact orders will also appear in this paper. However, the general comparison theorem is not necessary because we will reduce relevant invariants with several negative contact orders to invariants without negative contact order.  
\end{remark}

\section{The proper Landau--Ginzburg potential from intrinsic mirror symmetry}\label{sec:intrinsic}

\subsection{The Landau--Ginzburg potential}
A tropical view of the Landau--Ginzburg potential is given in \cite{CPS} using the toric degeneration approach to mirror symmetry. We consider the intrinsic mirror symmetry construction instead and focus on the case when the Landau--Ginzburg potential is proper.

Following intrinsic mirror symmetry \cite{GS19}, one considers a maximally unipotent degeneration $g:Y\rightarrow S$ of the smooth pair $(X,D)$. The mirror of $X\setminus D$ is constructed as the projective spectrum of the degree zero part of the relative quantum cohomology of $(Y,D^\prime)$, where $D^\prime$ is certain divisor of $Y$ that includes $g^{-1}(0)$. 

Let $(B^\prime,\mathscr P, \varphi)$ be the dual intersection complex or the fan picture of the degeneration. Recall that $B^\prime$ is an integral affine manifold with finite polyhedral decomposition $\mathscr P$ and a multi-valued strictly convex piecewise linear function $\varphi$.
An \emph{asymptotic direction} is an integral tangent vector of a one-dimensional unbounded cell in $(B^\prime,\mathscr P, \varphi)$ that points in the unbounded direction. 

\begin{definition}
The dual intersection complex $(B^\prime, \mathscr P)$ is asymptotically cylindrical if 
\begin{itemize}
    \item $B^\prime$ is non-compact.
    \item For every polyhedron $\sigma$ in $\mathscr P$, all of the unbounded one-faces of $\sigma$ are parallel with respect to the affine structure on $\sigma$.
\end{itemize} 
\end{definition}

We consider the case when $D$ is smooth. Hence, $(B^\prime, \mathscr P)$ is asymptotically cylindrical and $B^\prime$ has one unbounded direction $m_{\on{out}}$. We choose $\phi$ such that $\phi(m_{\on{out}})=1$ on all unbounded cells. 
%Let $\mathscr S_\infty$ be the consistent wall structure defined by $(B, \mathscr P, \phi)$. Then, for each broken line $\mathbf b$ for $\mathscr S_\infty$, the asymptotic direction $m_{\mathbf b}$ is parallel to $m_{\on{out}}$.

For the smooth pair $(X,D)$, one can also consider its relative quantum cohomology. Let $QH^0_{\on{log}}(X,D)$ be the degree zero subalgebra of the relative quantum cohomology ring $\on{QH}^*_{\on{log}}(X,D)$ of a pair $(X,D)$. Let $S$ be the dual intersection complex of $D$. Let $B$ be the cone over $S$ and $B(\mathbb Z)$ be the set of integer points of $B$. Since  $D$ is smooth, $B(\mathbb Z)$ is the set of nonnegative integers. The set
\[
\{\vartheta_p\}, p\in B(\mathbb Z)
\]
of theta functions form a canonical basis of $\on{QH}^0_{\on{log}}(X,D)$. Moreover, theta functions satisfy the following multiplication rule
\begin{align}\label{theta-func-multi}
\vartheta_{p_1}\star \vartheta_{p_2}=\sum_{r\geq 0, \beta}N_{p_1,p_2,-r}^{\beta} \vartheta_r.
\end{align}

%Let $x=t^{(-m_{\on{out}},-1)}$ where $m_{\on{out}}$........
\if{
Given a point $P$ on the dual intersection complex, the theta function can be defined by
\[
\vartheta_p(P)=\sum_{\mathfrak b}a_{\mathfrak b} m^{\mathfrak b},
\]
where the sum is over all broken lines $\mathbf b$ with asymptotic monomial $p$ and $\mathbf b(0)=P$. Then we have the following relation
\[
N_{p_1,p_2,-r}^{\beta}=\sum_{(\mathfrak a, \mathfrak b)} a_{\mathfrak a}a_{\mathfrak b}.
\]
}\fi

Recall that the structure constants $N^{\beta}_{p_1,p_2,-r}$ are defined as the invariants of $(X,D)$ with two ``inputs'' with positive contact orders given by $p_1, p_2\in B(\mathbb Z)$, one ``output'' with negative contact order given by $-r$ such that $r\in B(\mathbb Z)$, and a point constraint for the punctured point. Namely,
\begin{align}\label{def-stru-const}
    N^{\beta}_{p_1,p_2,-r}=\langle [1]_{p_1},[1]_{p_2},[\on{pt}]_{-r}\rangle_{0,3,\beta}^{(X,D)}.
\end{align}

We learnt about the following intrinsic mirror symmetry construction of the proper Landau--Ginzburg potential from Mark Gross (see \cite{GS19}*{Construction 1.19}). 

\begin{construction}
We recall the maximally unipotent degeneration $g:Y\rightarrow S$ and the pair $(Y,D^\prime)$ from the intrinsic mirror construction of the mirror $X^\vee$. Following the mirror construction of \cite{GS19}*{Construction 1.19}, the degree zero part of the graded ring of theta functions in $QH^0_{\log}(Y,D^\prime)$ agrees with $QH^0_{\log}(X,D)$. The base of the Landau--Ginzburg mirror of $(X,D)$ is $\on{Spec}QH^0_{\log}(X,D)=\mathbb A^1$ and the superpotential is $W=\vartheta_{1}$, the unique primitive theta function of $QH^0_{\log}(X,D)$. 
\end{construction}

Under this construction, to compute the proper Landau--Ginzburg potential, we just need to compute the theta function $\vartheta_1$. We would like to compute the structure constants $N^{\beta}_{p_1,p_2,-r}$ and then provide a definition of the theta functions, in terms of two-point relative invariants of $(X,D)$, which satisfy the multiplication rule (\ref{theta-func-multi}). The notion of broken lines will not be mentioned here.

\subsection{Structure constants}

We first express the structure constants in terms of two-point relative invariants.

\begin{prop}\label{prop-struc-const}
Let $(X,D)$ be a smooth log Calabi--Yau pair. Without loss of generality, we assume that $p_1\leq p_2$. Then the structure constants $N^{\beta}_{p_1,p_2,-r}$ can be written as two-point relative invariants (without negative contact):
\begin{align}\label{equ-punctured-2}
N^{\beta}_{p_1,p_2,-r}=\left\{
\begin{array}{cc}
     (p_1-r)\langle [\on{pt}]_{p_1-r}, [1]_{p_2}\rangle_{0,2,\beta}^{(X,D)}+ (p_2-r)\langle [\on{pt}]_{p_2-r}, [1]_{p_1}\rangle_{0,2,\beta}^{(X,D)} & \text{ if }0\leq r<p_1;
     \\ 
    (p_2-r)\langle [\on{pt}]_{p_2-r}, [1]_{p_1}\rangle_{0,2,\beta}^{(X,D)} & \text{ if }p_1\leq r<p_2;\\
    0 & \text{ if } r\geq p_2, r\neq p_1+p_2;\\
    1 & \text{ if } r=p_1+p_2.
\end{array}
\right.
\end{align}
  
\end{prop}
%\textcolor{red}{Maybe without the Calabi--Yau condition? not point insertion?}

\begin{proof}

We divide the proof into different cases. 

\begin{enumerate}
    \item $0<r<p_1:$
    
    We use the definition of relative Gromov--Witten invariants with negative contact orders in \cite{FWY}. Recall that $N^{\beta}_{p_1,p_2,-r}$ (\ref{def-stru-const}) is a relative invariant with one negative contact order. It can be written as
\begin{align}\label{def-one-negative}
\sum_{\mathfrak G\in \mathcal B_\Gamma}\frac{\prod_{e\in E}d_e}{|\Aut(E)|}\sum \langle \prod_{j\in S_1}\epsilon_j, \, |\,|[\on{pt}]_r,\eta\rangle^{\sim}_{\Gamma^0} \langle  \check{\eta},\prod_{j\in S_2}\epsilon_j\rangle^{\bullet, (X,D)}_{\Gamma_\infty},
\end{align}
where $S_1\sqcup S_2=\{1,2\}$, $\{\epsilon_j\}_{j=1,2}=\{[1]_{p_1},[1]_{p_2}\}$; the sum is over the cohomology weighted partition $\eta$ and the splitting $S_1\sqcup S_2=\{1,2\}$.

\begin{itemize}
    \item [(I) $|S_1|=\emptyset$:] 
    
    Then by the virtual dimension constraint on $\langle \, |\,|[\on{pt}]_r,\eta\rangle^{\sim}_{\Gamma^0}$, the insertions in $\eta$ must contain at least one element with insertion $[1]_k$, for some integer $k>0$. 
    
    Let $\pi:P:=\mathbb P_D(\mathcal O_D\oplus N_D)\rightarrow D$ be the projection map and $D_0$ and $D_\infty$ be the zero and infinity divisors of $P$. Let
\[
p: \bM_{\Gamma^0}^\sim(P,D_0\cup D_\infty)\rightarrow \bM_{0,m}(D,\pi_*(\beta_1))
\]
be the natural projection of the rubber map to $X$ and contracting the resulting unstable components. By \cite{JPPZ18}*{Theorem 2}, we have
\[
p_*[\bM_{\Gamma^0}^\sim(P,D_0\cup D_\infty)]^{\on{vir}}=[\bM_{0,m}(D,\pi_*(\beta_1))]^{\on{vir}}.
\]
The marking $[1]_k$ becomes the identity class $ 1\in H^*(D)$. Applying the string equation implies that the rubber invariant $\langle \, |\,|[\on{pt}]_r,\eta\rangle^{\sim}_{\Gamma^0}$ vanishes unless $\pi_*(\beta_1)=0$ and $m=3$. However, $\pi_*(\beta_1)=0$ implies that $D_0\cdot \beta_1=D_\infty\cdot \beta_1$. On the other hand, there is no relative marking at $D_0$. Therefore $D_0\cdot \beta_1=0$. We also know that $D_\infty\cdot \beta_1>0$ because $\eta$ is not empty. This is a contradiction. Therefore, we can not have $|S_1|=\emptyset$.

    \item [(II) $|S_1|\neq \emptyset$:] 

In this case,  for $\langle \prod_{j\in S_1}\epsilon_j, \, |\,|[\on{pt}]_r,\eta\rangle^{\sim}_{\Gamma^0} $, the relative insertion $\prod_{j\in S_1}\epsilon_j$ at $D_0$ is not empty. That is, it must contain at least one of $[1]_{p_1},[1]_{p_2}$.

Again, we consider the natural projection
\[
p: \bM_{\Gamma^0}^\sim(P,D_0\cup D_\infty)\rightarrow \bM_{0,m}(D,\pi_*(\beta_1)).
\]
Since $\prod_{j\in S_1}\epsilon_j$ must contain at least one of $[1]_{p_1},[1]_{p_2}$, by the projection formula and the string equation, 
\[
\langle \prod_{j\in S_1}\epsilon_j, \, |\,|[\on{pt}]_r,\eta\rangle^{\sim}_{\Gamma^0} 
\]
vanishes unless $\pi_*(\beta_1)=0$ and $m=3$. Note that $\pi_*(\beta_1)=0$ implies $D_0\cdot \beta=D_\infty\cdot \beta$, for any effective curve class $\beta$ of $P$. We recall that we assume that $r<p_1\leq p_2$, therefore $\eta$ must contain at least one markings with positive contact order. Then $\prod_{j\in S_1}\epsilon_j$ must contain exactly one of $[1]_{p_1},[1]_{p_2}$ when $r<p_1$. Therefore, $\eta$ contains exactly one element $[1]_{p_1-r}$ or $[1]_{p_2-r}$ respectively.  Hence, (\ref{def-one-negative}) is the sum of the following two invariants
\[
(p_1-r)\langle [\on{pt}]_{p_1-r}, [1]_{p_2}\rangle_{0,2,\beta}^{(X,D)}, \quad \text{and} \quad (p_2-r)\langle [\on{pt}]_{p_2-r},[1]_{p_1}\rangle_{0,2,\beta}^{(X,D)},
\]
which are exactly the invariants that appear on the RHS of (\ref{equ-punctured-2}) when $r<p_1$. 
\end{itemize}

\item $r=0$: 

In this case, there are no negative contacts. We can require the marking with the point insertion $[\on{pt}]_0$ maps to $D$. Consider the degeneration to the normal cone of $D$ and apply the degeneration formula. After applying the rigidification lemma \cite{MP}*{Lemma 2}, we also obtain the formula (\ref{def-one-negative}) with $r=0$. Then the rest of the proof is the same as the case when $0<r<p_1$.

\item $p_1\leq r< p_2$: 

We again have the formula (\ref{def-one-negative}) as in the first case. The difference is that we can not have $\prod_{j\in S_1}\epsilon_j=[1]_{p_1}$ because this will imply that $\eta$ contains the non-positive contact order element $[1]_{p_1-r}$. Therefore, we have 
\[
(p_2-r)\langle [\on{pt}]_{p_2-r}, [1]_{p_1}\rangle_{0,2,\beta}^{(X,D)}.
\]

\item $r\geq p_2$ and $r\neq p_1+p_2$: 

Similar to the previous case, we can not have $\prod_{j\in S_1}\epsilon_j=[1]_{p_1}$ or $\prod_{j\in S_1}\epsilon_j=[1]_{p_2}$. The invariant is $0$.

\item $p_1+p_2=r$: 

In this case, we can have $\eta$ to be empty. Then there is no $\Gamma^\infty$ and the curves entirely lie in $D$. Therefore, there is only one rubber integral. The invariant is just $1$.

\end{enumerate}

\end{proof}

\begin{remark}
A special case of Proposition \ref{prop-struc-const} also appears in \cite{Graefnitz2022}*{Theorem 2} for del Pezzo surfaces via tropical correspondence. Our result here uses the definition of \cite{FWY} for punctured invariants and it works for all dimensions and $X$ is not necessarily Fano.
\end{remark}

Later, we will also need to consider invariants of the following form:
\[
\langle [1]_{p_1},[\on{pt}]_{p_2},[1]_{-r}\rangle_{0,3,\beta}^{(X,D)}.
\]

The proof of the following identity is similar to the proof of Proposition \ref{prop-struc-const}. 
\begin{prop}\label{prop-theta-2}
Let $(X,D)$ be a smooth log Calabi--Yau pair. If $p_1\leq p_2$, then
\begin{align*}
 &\langle [1]_{p_1},[\on{pt}]_{p_2},[1]_{-r}\rangle_{0,3,\beta}^{(X,D)}\\
=&\left\{
\begin{array}{cc}
     (p_1-r)\langle [\on{pt}]_{p_2}, [1]_{p_1-r}\rangle_{0,2,\beta}^{(X,D)}+ (p_2-r)\langle  [\on{pt}]_{p_2-r}, [1]_{p_1}\rangle_{0,2,\beta}^{(X,D)} & \text{ if }0\leq r<p_1;
     \\ 
    (p_2-r)\langle [\on{pt}]_{p_2-r}, [1]_{p_1}\rangle_{0,2,\beta}^{(X,D)} & \text{ if }p_1\leq r<p_2;\\
    0 & \text{ if } r\geq p_2, r\neq p_1+p_2;\\
    1 &\text{ if } r=p_1+p_2.
\end{array}
\right.
\end{align*}
If $p_2\leq p_1$, then
\begin{align*}
 &\langle [1]_{p_1},[\on{pt}]_{p_2},[1]_{-r}\rangle_{0,3,\beta}^{(X,D)}\\
=&\left\{
\begin{array}{cc}
     (p_1-r)\langle [\on{pt}]_{p_2}, [1]_{p_1-r}\rangle_{0,2,\beta}^{(X,D)}+ (p_2-r)\langle  [\on{pt}]_{p_2-r}, [1]_{p_1}\rangle_{0,2,\beta}^{(X,D)} & \text{ if }0\leq r<p_2;
     \\ 
    (p_1-r)\langle  [\on{pt}]_{p_2}, [1]_{p_1-r}\rangle_{0,2,\beta}^{(X,D)} & \text{ if }p_2\leq r<p_1;\\
    0 & \text{ if } r\geq p_1, r\neq p_1+p_2;\\
    1 & \text{ if } r=p_1+p_2.
\end{array}
\right.
\end{align*}
\end{prop}

\subsection{Theta functions}

Now we define the theta function in terms of two-point relative Gromov--Witten invariants of $(X,D)$.
\begin{definition}\label{def-theta-func}
Write $x=z^{(-m_{\on{out}},-1)}$ and $t=z^{(0,1)}$. For $p\geq 1$, the theta function is
\begin{align}\label{theta-func-def}
\vartheta_p:=x^{-p}+\sum_{n=1}^{\infty}nN_{n,p}t^{n+p}x^n,
\end{align}
where 
\[
N_{n,p}=\sum_{\beta} \langle [\on{pt}]_n,[1]_p\rangle_{0,2,\beta}^{(X,D)}.
\]
\end{definition}

We also write
\[
N_{p_1,p_2,-r}:=\sum_\beta N_{p_1,p_2,-r}^{\beta}.
\]

To justify the definition of the theta function, we need to show that this definition satisfies the multiplication rule (\ref{theta-func-multi}). Plug-in $(\ref{theta-func-def})$ to $\vartheta_{p_1}\star \vartheta_{p_2}$, we have
\begin{align}\label{theta-p-q}
\notag    \vartheta_{p_1}\star \vartheta_{p_2}=&(x^{-p_1}+\sum_{m=1}^{\infty}mN_{m,p_1}t^{m+p_1}x^m)(x^{-p_2}+\sum_{n=1}^{\infty}nN_{n,p_2}t^{n+p_2}x^n)\\
    =&x^{-(p_1+p_2)}+\sum_{n=1}^{\infty}nN_{n,p_2}t^{n+p_2}x^{n-p_1}+\sum_{m=1}^{\infty}mN_{m,p_1}t^{m+p_1}x^{m-p_2}\\
  \notag  &+\sum_{m=1}^{\infty}\sum_{n=1}^{\infty}mnN_{m,p_1}N_{n,p_2}t^{m+p_1+n+p_2}x^{m+n}.
\end{align}
On the other hand, we have 
\begin{align}\label{theta-r}
\notag \sum_{r\geq 0, \beta}N_{p_1,p_2,-r}^{\beta}t^\beta \vartheta_r&=\sum_{r\geq 0}N_{p_1,p_2,-r}t^{p_1+p_2-r}\vartheta_r\\
&=\sum_{r\geq 0}N_{p_1,p_2,-r}t^{p_1+p_2-r}(x^{-r}+\sum_{k=1}^{\infty}kN_{k,r}t^{k+r}x^k),
\end{align}
where the second line follows from (\ref{theta-func-def}). Note that $N_{k,r}=0$ when $r=0$ by the string equation. 

By Proposition \ref{prop-struc-const}, it is straightforward that the coefficients of $x^k$, for $k\leq 0$, of (\ref{theta-p-q}) and (\ref{theta-r}) are the same: without loss of generality, we assume that $p_1\leq p_2$. Then, we have the following cases.
\begin{itemize}
    \item $k\leq -p_2$, and $k\neq -p_1-p_2$: we see that the coefficient of $\vartheta_{p_1}\star \vartheta_{p_2}$ in (\ref{theta-p-q}) is zero. The corresponding coefficient $N_{p_1,p_2,k}$ in (\ref{theta-r}) is also zero because of Proposition \ref{prop-struc-const}.
    
    \item $k=-p_1-p_2$: we see that the coefficient of $\vartheta_{p_1}\star \vartheta_{p_2}$ in (\ref{theta-p-q}) is $1$. The corresponding coefficient $N_{p_1,p_2,k}$ in (\ref{theta-r}) is also $1$ because of Proposition \ref{prop-struc-const}.
    
    \item $-p_2<k\leq -p_1$: the coefficient of $\vartheta_{p_1}\star \vartheta_{p_2}$ in (\ref{theta-p-q}) is $(p_2+k)N_{p_2+k,p_1}$. By Proposition \ref{prop-struc-const}, the corresponding coefficient $N_{p_1,p_2,k}$ in (\ref{theta-r}) is:
    \[
    N_{p_1,p_2,k}=(p_2+k)N_{p_2+k,p_1}.
    \]
    
    \item $-p_1<k\leq 0$: the coefficient of $\vartheta_{p_1}\star \vartheta_{p_2}$ in (\ref{theta-p-q}) is
    \[
    (p_1+k)N_{p_1+k,p_2}+(p_2+k)N_{p_2+k,p_1}.
    \]
    This coincides with the corresponding coefficient $N_{p_1,p_2,k}$ by Proposition \ref{prop-struc-const}. 
\end{itemize}

For the coefficients of $x^k$, for $k>0$, the coefficients also match because of the following result.

\begin{prop}\label{prop-wdvv}
Let $(X,D)$ be a smooth log Calabi--Yau pair. We have
\begin{align}\label{identity-wdvv}
(k+p_1)N_{k+p_1,p_2}+(k+p_2)N_{k+p_2,p_1}+\sum_{m,n>0, m+n=k}mnN_{m,p_1}N_{n,p_2}=\sum_{r>0}N_{p_1,p_2,-r}kN_{k,r}.
\end{align}
\end{prop}

\begin{proof}
It can be proved using the WDVV equation of relative Gromov--Witten theory in \cite{FWY}*{Proposition 7.5}. Set
\[
[\alpha_1]_{i_1}=[1]_{p_1}, [\alpha_2]_{i_2}=[1]_{p_2}, [\alpha_3]_{i_3}=[\on{pt}]_{k+p_1+p_2}, [\alpha_4]_{i_4}=[1]_{-p_1-p_2}.
\]
Then the WDVV equation states that
\begin{align}\label{wdvv}
\sum\langle [1]_{p_2},[\on{pt}]_{k+p_1+p_2},[\gamma]_{-i}\rangle_{0,\beta_1,3}^{(X,D)} \langle [\gamma^\vee]_{i}, [1]_{p_1},[1]_{-p_1-p_2} \rangle_{0,\beta_2,3}^{(X,D)} \\
\notag =
\sum\langle [1]_{p_1},[1]_{p_2},[\gamma]_{-i}\rangle_{0,\beta_1,3}^{(X,D)}  \langle [\gamma^\vee]_{i}, [\on{pt}]_{k+p_1+p_2},[1]_{-p_1-p_2} \rangle_{0,\beta_2,3}^{(X,D)} ,
\end{align}
where each sum is over the curve class $\beta$ such that $D\cdot \beta=p_1+p_2+k$, all splittings of $\beta_1+\beta_2=\beta$ and the dual bases  $\{[\gamma]_{-i}\}$ and $\{[\gamma^\vee]_{i}\}$ of $\mathfrak H$.

\begin{enumerate}
    \item[\textbf{(I)}] We first consider the LHS of the WDVV equation (\ref{wdvv}):
\begin{align}\label{wdvv-lhs}
    \sum\langle [1]_{p_2},[\on{pt}]_{k+p_1+p_2},[\gamma]_{-i}\rangle_{0,\beta_1,3}^{(X,D)}  \langle [\gamma^\vee]_{i}, [1]_{p_1},[1]_{-p_1-p_2} \rangle_{0,\beta_2,3}^{(X,D)}. 
\end{align}

 We analyze the invariant
\[
\langle [\gamma^\vee]_{i}, [1]_{p_1},[1]_{-p_1-p_2} \rangle_{0,\beta_2,3}^{(X,D)}
\]
in (\ref{wdvv-lhs}). 
\begin{itemize}
    \item[(i) $i<0$:] we claim that the invariant vanishes. 
By the virtual dimension constraint, $\deg(\gamma^\vee)=\dim_{\mathbb C}X-2$.

We apply the definition of relative Gromov—Witten invariants with negative contact orders in Section \ref{sec-rel-general}. The marking with negative contact order $[1]_{-p_1-p_2}$ is distributed to a rubber space. The marking becomes a relative marking at $D_\infty$ with insertion $[1]_{p_1+p_2}$.

We further divide it into two cases.
\begin{itemize}
    \item [(Case 1):] the first marking  and the third marking are distributed to different rubber spaces. Then the class $\mathfrak c_\Gamma$ is trivial. We consider the rubber moduli space $\overline{\mathcal M}_{\Gamma^0_v}^\sim(P,D_0\cup D_\infty)$ where the third marking is distributed to and pushforward this rubber moduli space to the moduli space $\overline{M}_{0,m}(D,\pi_*\beta_v)$ of stable maps to $D$. The marking  with $[1]_{p_1+p_2}$ becomes the identity class $1\in H^*(D)$. Apply the string equation, we see that the rubber invariant vanishes unless $\pi_*(\beta_v)=0$ and $m=3$. However $\pi_*(\beta_v)=0$ implies that $D_0\cdot \beta_v=D_\infty\cdot \beta_v$. This is not possible because, based on the insertions of the markings, we must have $D_\infty\cdot \beta_v\geq p_1+p_2>D_0\cdot \beta_v$.
    \item [(Case 2):]  the first marking and the third marking are distributed to the same rubber space. The class $\mathfrak c_\Gamma$ is a sum of descendant classes of degree one. By the virtual dimension constraint, $\eta$ must contain at least one element with insertion $[1]_k$ for some positive integer $k$. Pushing forward to the moduli space of stable maps to $D$ and applying the string equation twice, we again conclude that the invariant vanishes as in (Case 1). 
\end{itemize}

\item [(ii) $i\geq 0$:]
 The invariants in (\ref{wdvv-lhs}) are genus zero $3$-point relative invariants of $(X,D)$ with one negative contact order. Therefore, the virtual dimensions of the moduli spaces are $(\dim_{\mathbb C} X-1)$. By the virtual dimension constraint, we must have
\[
[\gamma]_{-i}=[1]_{-i}, \quad [\gamma^\vee]_{i}=[\on{pt}]_{i}.
\]

By Proposition \ref{prop-theta-2}, $ \langle [\on{pt}]_{i}, [1]_{p_1},[1]_{-p_1-p_2} \rangle_{0,\beta_2,3}^{(X,D)} $ vanishes unless $i>p_1+p_2$ or $i=p_2$. 
\begin{itemize}
    \item When $i=p_2$, we have the term
\[
\langle [1]_{p_2},[\on{pt}]_{k+p_1+p_2},[1]_{-i}\rangle_{0,\beta_1,3}^{(X,D)}  \langle [\on{pt}]_{i}, [1]_{p_1},[1]_{-p_1-p_2} \rangle_{0,\beta_2,3}^{(X,D)}=(k+p_1)\langle [\on{pt}]_{k+p_1}, [1]_{p_2}\rangle_{0,\beta,2}^{(X,D)}
\]
in (\ref{wdvv-lhs}), by Proposition \ref{prop-theta-2}.
\item When $i>p_1+p_2$, we have
\[
\langle [\on{pt}]_{i}, [1]_{p_1},[1]_{-p_1-p_2} \rangle_{0,\beta_2,3}^{(X,D)}=(i-p_1-p_2)\langle [\on{pt}]_{i-p_1-p_2}, [1]_{p_1}\rangle_{0,\beta_2,2}^{(X,D)},
\]
 by Proposition \ref{prop-theta-2}.
\end{itemize}

\end{itemize}

Similarly, by Proposition \ref{prop-theta-2}, $\langle [1]_{p_2},[\on{pt}]_{k+p_1+p_2},[1]_{-i}\rangle_{0,\beta_1,3}^{(X,D)}$ vanishes unless $i<k+p_1+p_2$ or $i=k+p_1+2p_2$. 
\begin{itemize}
    \item When $i=k+p_1+2p_2$ we have the term
\[
\langle [1]_{p_2},[\on{pt}]_{k+p_1+p_2},[1]_{-i}\rangle_{0,\beta_1,3}^{(X,D)}  \langle [\on{pt}]_{i}, [1]_{p_1},[1]_{-p_1-p_2} \rangle_{0,\beta_2,3}^{(X,D)}=(k+p_2)\langle [\on{pt}]_{k+p_2}, [1]_{p_1}\rangle_{0,\beta,2}^{(X,D)},
\]
in (\ref{wdvv-lhs}), by Proposition \ref{prop-theta-2}.
\item When $i<k+p_1+p_2$, we have
\[
\langle [1]_{p_2},[\on{pt}]_{k+p_1+p_2},[1]_{-i}\rangle_{0,\beta_1,3}^{(X,D)} =(k+p_1+p_2-i)\langle [\on{pt}]_{k+p_1+p_2-i}, [1]_{p_2}\rangle_{0,\beta_2,2}^{(X,D)},
\]
 by Proposition \ref{prop-theta-2}.
\end{itemize}

We summarize the above analysis of (\ref{wdvv-lhs}) in terms of $i$:
\begin{itemize}
    \item $i<0$, the summand is $0$.
    \item $i=p_2$, the summand is
    \[
    (k+p_1)\langle [\on{pt}]_{k+p_1}, [1]_{p_2}\rangle_{0,\beta,2}^{(X,D)}.
    \]
    \item $i=k+p_1+2p_2$, the summand is
    \[
    (k+p_2)\langle [\on{pt}]_{k+p_2}, [1]_{p_1}\rangle_{0,\beta,2}^{(X,D)}.
    \]
    \item $p_1+p_2 <i<k+p_1+p_2$, the summand is
    \[
    (k+p_1+p_2-i)\langle [\on{pt}]_{k+p_1+p_2-i}, [1]_{p_2}\rangle_{0,\beta_2,2}^{(X,D)}(i-p_1-p_2)\langle [\on{pt}]_{i-p_1-p_2}, [1]_{p_1}\rangle_{0,\beta_2,2}^{(X,D)}.
    \]
\end{itemize}

Set $m:=i-p_1-p_2$ and $n:=k+p_1+p_2-i$. Then (\ref{wdvv-lhs}) becomes 
\[
(k+p_1)N_{k+p_1,p_2}+(k+p_2)N_{k+p_2,p_1}+\sum_{m,n>0, m+n=k}mnN_{m,p_1}N_{n,p_2}.
\]
\item[\textbf{(II)}] Now we look at the RHS of the WDVV equation (\ref{wdvv}):
\begin{align}\label{wdvv-rhs}
    \sum\langle [1]_{p_1},[1]_{p_2},[\gamma]_{-i}\rangle_{0,\beta_1,3}^{(X,D)}  \langle [\gamma^\vee]_{i}, [\on{pt}]_{k+p_1+p_2},[1]_{-p_1-p_2} \rangle_{0,\beta_2,3}^{(X,D)}.
\end{align}
\begin{itemize}
    \item If $i<0$, then the invariant $\langle [\gamma^\vee]_{i}, [\on{pt}]_{k+p_1+p_2},[1]_{-p_1-p_2} \rangle_{0,\beta_2,3}^{(X,D)}$ is a genus zero three-point relative invariant with two negative contact orders. The virtual dimension of the moduli space is $(n-2)$. On the other hand, the second marking has a point insertion:
\[
\deg([\on{pt}]_{k+p_1+p_2})=n-1>n-2.
\]
It is a contradiction. 
\item If $i=0$, by the virtual dimension constraint, we must have
\[
[\gamma^\vee]_{i}=[1]_0.
\]
The string equation implies the invariant is zero. 
\item If $i>0$, we must have
\[
[\gamma^\vee]_{i}=[1]_{r}, \quad \text{and } [\gamma]_{-i}=[\on{pt}]_{-r} \quad \text{for } r:=i>0.
\]
We have
\[
\langle [1]_{p_1},[1]_{p_2},[\gamma]_{-i}\rangle_{0,\beta_1,3}^{(X,D)}=\langle [1]_{p_1},[1]_{p_2},[\on{pt}]_{-r}\rangle_{0,\beta_1,3}^{(X,D)}. 
\]
%Recall that, by Proposition \ref{prop-struc-const}, $\langle [1]_p,[1]_q,[\on{pt}]_{-r}\rangle_{0,\beta_1,3}^{(X,D)}=0$ if $r\geq \max\{p,q\}$.
By Proposition \ref{prop-theta-2},
\[
\langle [1]_{r}, [\on{pt}]_{k+p_1+p_2},[1]_{-p_1-p_2} \rangle_{0,\beta_2,3}^{(X,D)}=k\langle [1]_{r}, [\on{pt}]_{k} \rangle_{0,\beta_2,2}^{(X,D)}.
\]
\end{itemize}

Therefore, (\ref{wdvv-rhs}) becomes
\[
\sum_{r>0}N_{p_1,p_2,-r}kN_{k,r}.
\]
\end{enumerate}
This completes the proof.
\end{proof}

\begin{remark}

We can consider the special case when $p_1=1$, then the LHS of (\ref{identity-wdvv}) is
\[
(k+1)N_{k+1,p_2}+(k+p_2)N_{k+p_2,1}+\sum_{m,n>0, m+n=k}mnN_{m,1}N_{n,p_2}
\]
and the RHS of (\ref{identity-wdvv}) is
\begin{align*}
&\sum_{r>0}N_{1,p_2,-r}kN_{k,r}\\
=& kN_{k,p_2+1}+ \sum_{r=1}^{p_2-1}(p_2-r)N_{p_2-r,1}kN_{k,r},
\end{align*}
by Proposition \ref{prop-struc-const}. Identity (\ref{identity-wdvv}) becomes
\[
(k+1)N_{k+1,p_2}+(k+p_2)N_{k+p_2,1}+\sum_{m,n>0, m+n=k}mnN_{m,1}N_{n,p_2}
=kN_{k,p_2+1}+ \sum_{r=1}^{p_2-1}(p_2-r)N_{p_2-r,1}kN_{k,r}.
\]
If we further specialize to the case of toric del Pezzo surfaces with smooth divisors, we recover \cite{GRZ}*{Lemma 5.3}. Here, we give a direct explanation of Identity (\ref{identity-wdvv}) in terms of the WDVV equation of the relative Gromov--Witten theory in \cite{FWY}.
\end{remark}

\section{A mirror theorem for smooth pairs}

\subsection{Relative mirror theorem}

Let $X$ be a smooth projective variety. Let $\{p_i\}_{i=1}^r$ be an integral, nef basis of $H^2(X)$. For the rest of the paper, we assume that $D$ is nef.

Recall that the $J$-function for absolute Gromov--Witten theory of $X$ is
\[
J_{X}(\tau,z)=z+\tau+\sum_{\substack{(\beta,l)\neq (0,0), (0,1)\\ \beta\in \on{NE(X)}}}\sum_{\alpha}\frac{q^{\beta}}{l!}\left\langle \frac{\phi_\alpha}{z-\psi},\tau,\ldots, \tau\right\rangle_{0,1+l, \beta}^{X}\phi^{\alpha},
\]
where $\tau=\tau_{0,2}+\tau^\prime\in  H^*(X)$; $\tau_{0,2}=\sum_{i=1}^r p_i \log q_i\in H^2(X)$; $\tau^\prime\in H^*(X)\setminus H^2(X)$; $\on{NE(X)}$ is the cone of effective curve classes in $X$; $\{\phi_\alpha\}$ is a basis of $H^*(X)$; $\{\phi^\alpha\}$ is the dual basis under the Poincar\'e pairing. We can decompose the $J$-function as follows
\[
J_{X}(\tau,z)=\sum_{\beta\in \on{NE(X)}}J_{X,\beta}(\tau,z)q^\beta.
\]

The $J$-function of the smooth pair $(X,D)$ is defined similarly.
We first define 
\[
\HH_0:=H^*(X) \text{ and }\HH_i:=H^*(D) \text{ if }i\in \mathbb Z \setminus \{0\}.
\]
The ring of insertions (state space) of relative Gromov--Witten theory is defined as
\[
\HH:=\bigoplus\limits_{i\in\mathbb Z}\HH_i.
\]
Each $\HH_i$ naturally embeds into $\HH$. For an element $\gamma\in \HH_i$, we denote its image in $\HH$ by $[\gamma]_i$. Define a pairing on $\HH$ by the following.
\begin{equation}\label{eqn:pairing}
\begin{split}
([\gamma]_i,[\delta]_j) = 
\begin{cases}
0, &\text{if } i+j\neq 0,\\
\int_X \gamma\cup\delta, &\text{if } i=j=0, \\
\int_D \gamma\cup\delta, &\text{if } i+j=0, i,j\neq 0.
\end{cases}
\end{split}
\end{equation}
The pairing on the rest of the classes is generated by linearity.

\begin{defn}\label{def-relative-J-function}
Let $X$ be a smooth projective variety and $D$ be a smooth nef divisor, the $J$-function for the pair $(X,D)$ is defined as
\[
J_{(X,D)}(\tau,z)=z+\tau+\sum_{\substack{(\beta,l)\neq (0,0), (0,1)\\ \beta\in \on{NE(X)}}}\sum_{\alpha}\frac{q^{\beta}}{l!}\left\langle \frac{\phi_\alpha}{z-\bar{\psi}},\tau,\ldots, \tau\right\rangle_{0,1+l, \beta}^{(X,D)}\phi^{\alpha},
\]
where $\tau=\tau_{0,2}+\tau^\prime\in  H^*(X)$; $\tau_{0,2}=\sum_{i=1}^r p_i \log q_i\in H^2(X)$; $\tau^\prime\in \mathfrak H\setminus H^2(X)$; $\{\phi_\alpha\}$ is a basis of the ambient part of $\mathfrak H$; $\{\phi^\alpha\}$ is the dual basis under the Poincar\'e pairing. 
\end{defn}

%\begin{remark}
%There is a difference between relative Gromov--Witten theory and absolute Gromov--Witten theory. When $\beta=0$.........
%\end{remark}

\begin{defn}\label{def-relative-I-function}
The (non-extended) $I$-function of the smooth pair $(X,D)$ is 
\[
I_{(X,D)}(y,z)=\sum_{\beta\in \on{NE(X)}}J_{X,\beta}(\tau_{0,2},z)y^\beta\left(\prod_{0<a\leq D\cdot \beta-1}(D+az)\right)[1]_{-D\cdot \beta},
\]
where $\tau_{0,2}\in H^2(X)$.
\end{defn}

\begin{theorem}[\cite{FTY}, Theorem 1.4]\label{thm:rel-mirror}
Let $X$ be a smooth projective variety and $D$ be a smooth nef divisor such that $-K_X-D$ is nef. Then the $I$-function $I_{(X,D)}(y,\tau,z)$ coincides with the $J$-function $J_{(X,D)}(q,z)$ via change of variables, called the relative mirror map.
\end{theorem}

The relative mirror theorem also holds for the extended $I$-function of the smooth pair $(X,D)$. For the purpose of this paper, we only write down the simplest case when the extended data $S$ is the following:
\[
S:=\{1\}.
\]

\begin{defn}\label{def-relative-I-function-extended}
The $S$-extended $I$-function of $(X,D)$ is defined as follows. 
\[
I_{(X,D)}^{S}(y,x_1,z)=I_++I_-,
\]
where
\begin{align*}
I_+:=&\sum_{\substack{\beta\in \on{NE}(X),k\in \mathbb Z_{\geq 0}\\  k<D\cdot \beta} }J_{X, \beta}(\tau_{0,2},z)y^{\beta}\frac{ x_1^{k}}{z^{k}k!}\frac{\prod_{0<a\leq D\cdot \beta}(D+az)}{D+(D\cdot \beta-k)z}[{1}]_{-D\cdot \beta+k},
\end{align*}
and 
\begin{align*}
I_-:=&\sum_{\substack{\beta\in \on{NE}(X),k\in \mathbb Z_{\geq 0}\\ k\geq D\cdot \beta} }J_{X, \beta}(\tau_{0,2},z)y^{\beta}\frac{ x_1^{k}}{z^{k}k!}\left(\prod_{0<a\leq D\cdot \beta}(D+az)\right)[{ 1}]_{-D\cdot \beta+k}.
\end{align*}
\end{defn}

\begin{theorem}[\cite{FTY}, Theorem 1.5]\label{thm:rel-mirror-extended}
Let $X$ be a smooth projective variety and $D$ be a smooth nef divisor such that $-K_X-D$ is nef. Then the $S$-extended $I$-function $I^S_{(X,D)}(y,x_1,z)$ coincides with the $J$-function $J_{(X,D)}(q,z)$ via change of variables, called the relative mirror map.
\end{theorem}

Although the relative mirror theorem of \cite{FTY} has been used in the literature several times, the relative mirror map has never been studied in detail. We would like to provide a detailed description of the relative mirror map here.

We consider the extended $I$-function in Definition \ref{def-relative-I-function-extended} under the assumption that $D$ is nef and $-K_X-D$ is nef. The extended $I$-function can be expanded as follows
\begin{align*}
&I^S_{(X,D)}(y,x_1,z)\\
=&z+\sum_{i=1}^r p_i\log y_i+x_1[1]_{1}+\sum_{\substack{\beta\in \on{NE}(X)\\ D\cdot \beta \geq 2}}\langle [\on{pt}]\psi^{D\cdot \beta-2}\rangle_{0,1,\beta}^Xy^\beta (D\cdot \beta-1)![1]_{-D\cdot \beta}+\sum_{k=1}^{\infty}I_{-k}z^{-k},
\end{align*}
where the coefficient of $z^0$,  denoted by $\tau(y,x_1)$, is the relative mirror map:
\begin{align}\label{rel-mirror-map}
\tau(y,x_1)=\sum_{i=1}^r p_i\log y_i+x_1[1]_{1}+\sum_{\substack{\beta\in \on{NE}(X)\\ D\cdot \beta \geq 2}}\langle [\on{pt}]\psi^{D\cdot \beta-2}\rangle_{0,1,\beta}^Xy^\beta (D\cdot \beta-1)![1]_{-D\cdot \beta}.
\end{align}

The relative mirror theorem of \cite{FTY} states that
\[
J(\tau(y,x_1),z)=I(y,x_1,z).
\]
The function $\tau(y,x_1)$ is the mirror map computed from the extended $I$-function. We will refer to $\tau(y,x_1)$ as the \emph{extended relative mirror map}. The relative mirror map for the non-extended $I$-function is $\tau(y,0)$. We will refer to it as the \emph{relative mirror map} and denote it by $\tau(y)$.

To be able to compute invariants from the relative mirror theorem, we need to understand the invariants that appear in $J(\tau(y,x_1),z)$. In particular, we need to understand the following invariants in order to compute the theta function $\vartheta_1$.
\begin{itemize}
    \item Relative invariants with one positive contact order and several negative contact orders:
    \[
    \langle [1]_{-k_1},\cdots, [1]_{-k_l}, [\gamma]_{k_{l+1}} \bar{\psi}^a\rangle_{0,l+1,\beta}^{(X,D)},
    \]
    where
    \[
    D\cdot \beta=k_{l+1}-\sum_{i=1}^l k_i\geq 0, \text{ and } k_i>0.
    \]
    This is needed to understand the relative mirror map.
    \item Relative invariants with two positive contact orders and several negative contact orders of the following form:
    \[
    \langle [1]_1, [1]_{-k_1},\cdots, [1]_{-k_l}, [\gamma]_{k_{l+1}} \rangle_{0,l+2,\beta}^{(X,D)},
    \]
     \[
    D\cdot \beta-1=k_{l+1}-\sum_{i=1}^l k_i\geq 0, \text{ and } k_i>0.
    \]
    \item Degree zero relative invariants  with two positive contact orders and several negative contact orders of the following form:
    \[
    \langle [1]_1,[1]_{-k_1},\cdots, [1]_{-k_l},[\on{pt}]_{k_{l+1}}\rangle_{0,l+2,0}^{(X,D)}.
    \]
\end{itemize}

We will compute these invariants in the following sections.

\subsection{Relative invariants with several negative contact orders}
Based on the expression of relative mirror map in (\ref{rel-mirror-map}), to be able to compute relative invariants from the relative mirror theorem, we first need to study relative invariants with several insertions of $[1]_{-i}$ for $i\in \mathbb Z_{>0}$. We start with the case when $x=0$. That is, there is only one marking with positive contact order and no marking with insertion $[1]_1$. We would like to claim that the insertion $[1]_{-i}$ behaves like the divisor class $D$ in the sense that there is an analogous of the divisor equation as follows.
\begin{proposition}\label{prop-several-neg-1}
Given a curve class $\beta$, Let $k_i\in \mathbb Z_{>0}$ for $i\in \{1,\ldots, l+1\}$ such that
\[
D\cdot \beta=k_{l+1}-\sum_{i=1}^l k_i\geq 0.
\]
Then we have the following relation.
\begin{align}\label{identity-several-neg-1}
\langle [1]_{-k_1},\cdots, [1]_{-k_l}, [\gamma]_{k_{l+1}} \bar{\psi}^a\rangle_{0,l+1,\beta}^{(X,D)}=\langle [D]_0,\cdots, [D]_0, [\gamma]_{D\cdot \beta} \bar{\psi}^a\rangle_{0,l+1,\beta}^{(X,D)},
\end{align}
where $\gamma \in H^*(D)$.
\end{proposition}
\begin{proof}

\textbf{The base case I: $a=0$.}

In this case, there are no descendant classes. Then the identity becomes
\begin{align}\label{identity-several-neg-1-0}
\langle [1]_{-k_1},\cdots, [1]_{-k_l}, [\gamma]_{k_{l+1}} \rangle_{0,l+1,\beta}^{(X,D)}=\langle [D]_0,\cdots, [D]_0, [\gamma]_{D\cdot \beta} \rangle_{0,l+1,\beta}^{(X,D)}.
\end{align}
By Section \ref{sec-rel-general}, relative Gromov--Witten theory is defined as graph sums by gluing moduli spaces of relative stable maps with moduli space of rubber maps using fiber products. When there are more than one negative contact orders, the invariants are usually complicated and involve summation over different graphs as described in Section \ref{sec-rel-general}. But for invariants on the LHS of (\ref{identity-several-neg-1-0}), the situation is significantly simplified.

Every negative contact marking must be distributed to a rubber moduli $\bM^\sim_{\Gamma_i^0}(D)$ labelled by $\Gamma_i^0$. Since $D$ is a nef divisor in $X$, we have
\[
\int_{\beta_D}c_1(N_{D/X})\geq 0,
\]
for every effective curve class $\beta_D$ in $D$. Let $\beta_v$ be a curve class associated to a vertex $v$ in $\Gamma^0$, we must have
\[
D_0\cdot \beta_v-D_\infty\cdot \beta_v=\int_{\pi_*(\beta_v)}c_1(N_{D/X})\geq 0,
\]
where $\pi:P\rightarrow D$ is the projection map.
Therefore, the nefness of $D$ implies that, for each $\bM^\sim_{\Gamma_i^0}(D)$, the relative insertion at $D_0$ can not be empty. Hence, at least one of the positive contact markings on the LHS of (\ref{identity-several-neg-1-0}) must be distributed to $\bM^\sim_{\Gamma_i^0}(D)$. Since there is only one positive contact marking, there can only be one rubber moduli, denoted by $\bM^\sim_{\Gamma^0}(D)$. Therefore, all the negative contact markings, as well as the positive contact marking, are distributed to $\bM^\sim_{\Gamma^0}(D)$. The invariant can be written as
\[
\sum_{\mathfrak G\in \mathcal B_\Gamma}\frac{\prod_{e\in E}d_e}{|\Aut(E)|}\sum_\eta \langle \check{\eta}\rangle^{\bullet,\mathfrak c_\Gamma, (X,D)}_{\Gamma^\infty}\langle \eta,[1]_{k_1},\cdots, [1]_{k_l},| \, | [\gamma]_{k_{l+1}} \rangle^{\sim,\mathfrak c_{\Gamma}}_{\Gamma^0}, 
\]
where the superscript $\mathfrak c_\Gamma$ means capping with the class $\mathfrak c_\Gamma(X/D)$ in Definition \ref{def-rel-cycle}.

Let 
\[
p: \bM_{\Gamma^0}^\sim(P,D_0\cup D_\infty)\rightarrow \bM_{0,m}(D,\pi_*(\beta_1))
\]
be the natural projection of the rubber map to $X$ and contracting the resulting unstable components. By \cite{JPPZ18}*{Theorem 2}, we have
\[
p_*[\bM_{\Gamma^0}^\sim(P,D_0\cup D_\infty)]^{\on{vir}}=[\bM_{0,m}(D,\pi_*(\beta_1))]^{\on{vir}},
\]
where $\pi:P\rightarrow D$ is the projection map.

Note that there are $l$ identity classes $[1]$ and the degree of the class $\mathfrak c_\Gamma(X/D)$ is less or equal to $l-1$. We can apply the string equation $l$-times. Then the invariant 
\[
\langle \eta,[1]_{k_1},\cdots, [1]_{k_l},| \, | [\gamma]_{k_{l+1}} \rangle^{\sim,\mathfrak c_{\Gamma}}_{\Gamma^0}
\]
vanishes unless $\pi_*(\beta)=0$ and $\eta$ contains exactly one element. Moreover, $\eta$ needs to be $[\check{\gamma}]_{D\cdot \beta}$ and $\check{\eta}$ needs to be $[\gamma]_{D\cdot \beta}$. Therefore,
\[
\langle \check{\eta}\rangle^{\bullet,\mathfrak c_\Gamma, (X,D)}_{\Gamma^\infty}=\langle [\gamma]_{D\cdot \beta}\rangle_{0,1,\beta}^{(X,D)}.
\]
There is only one edge, hence
\[
\frac{\prod_{e\in E}d_e}{|\Aut(E)|}=D\cdot \beta.
\]

It remains to compute 
\[
\langle [\check{\gamma}]_{D\cdot \beta},[1]_{k_1},\cdots, [1]_{k_l},| \, | [\gamma]_{k_{l+1}} \rangle^{\sim,\mathfrak c_{\Gamma}}_{\Gamma^0}
\]
with $\pi_*(\beta)=0$. This is the same as the rubber invariant with the base being a point. Set $d=D\cdot \beta$. We claim that it coincides with the following relative Gromov--Witten invariants of $(\mathbb P^1,0\cup\infty)$ with negative contact orders 
\begin{align}\label{rel-inv-P-1}
\langle [1]_{d}, [1]_{-k_1},\cdots, [1]_{-k_l}, | \, | [1]_{k_{l+1}}\rangle_{0,l+2,d}^{(\mathbb P^1,0\cup\infty)}.
\end{align}
This is because one can run the above computation of the LHS of (\ref{identity-several-neg-1-0}) to the invariant (\ref{rel-inv-P-1}), we see that (\ref{rel-inv-P-1}) equals to
\[
d\langle [1]_d | \, | [1]_d \rangle_{0,l+2,d}^{(\mathbb P^1,0\cup\infty)} \langle [\check{\gamma}]_{D\cdot \beta},[1]_{k_1},\cdots, [1]_{k_l},| \, | [\gamma]_{k_{l+1}} \rangle^{\sim,\mathfrak c_{\Gamma}}_{\Gamma^0}.
\]
It is straightforward to compute that
\[
\langle [1]_d | \, | [1]_d \rangle_{0,l+2,d}^{(\mathbb P^1,0\cup\infty)}=\frac{1}{d}.
\]
This proves the claim that 
\[
\langle [\check{\gamma}]_{D\cdot \beta},[1]_{k_1},\cdots, [1]_{k_l},| \, | [\gamma]_{k_{l+1}} \rangle^{\sim,\mathfrak c_{\Gamma}}_{\Gamma^0}
\]
with $\pi_*(\beta)=0$ equals to (\ref{rel-inv-P-1}). The invariant (\ref{rel-inv-P-1}) has already been computed in \cite{KW}*{Proposition B.2}:
\[
\langle [1]_{d}, [1]_{-k_1},\cdots, [1]_{-k_l}, | \, | [1]_{k_{l+1}}\rangle_{0,l+2,d}^{(\mathbb P^1,0\cup\infty)}=d^{l-1}.
\]

Therefore the LHS of (\ref{identity-several-neg-1-0}) is

\[
(D\cdot \beta)^l \langle [\gamma]_{D\cdot \beta}\rangle_{0,1,\beta}^{(X,D)}=\langle [D]_0,\cdots, [D]_0, [\gamma]_{D\cdot \beta} \rangle_{0,l+1,\beta}^{(X,D)}
\]
by the divisor equation. 

%\textcolor{red}{Add more details to the above proof.}

\textbf{ The base case II: $l=1$.}

Then the LHS of (\ref{identity-several-neg-1-0}) is a relative invariant with one negative contact order. Similar to the proof of Proposition \ref{prop-struc-const}, the invariant is of the form,
\begin{align*}
\sum_{\mathfrak G\in \mathcal B_\Gamma}\frac{\prod_{e\in E}d_e}{|\Aut(E)|}\sum \langle [\gamma]_{k_2}\psi^a |\,|[1]_{k_1},\eta\rangle^{\sim}_{\Gamma^0} \langle  \check{\eta}\rangle^{\bullet, (X,D)}_{\Gamma_\infty}.
\end{align*}

For the RHS of (\ref{identity-several-neg-1-0}), we consider the degeneration to the normal cone of $D$ and apply the degeneration formula. After applying the rigidification lemma \cite{MP}*{Lemma 2}, the invariant is of the form
\begin{align*}
\sum_{\mathfrak G\in \mathcal B_\Gamma}\frac{\prod_{e\in E}d_e}{|\Aut(E)|}\sum \langle [\gamma]_{k_2-k_1}\psi^a |\,|[1]_{0},\eta\rangle^{\sim}_{\Gamma^0} \langle  \check{\eta}\rangle^{\bullet, (X,D)}_{\Gamma_\infty}.
\end{align*}
The only difference between the LHS of (\ref{identity-several-neg-1-0}) and the RHS of (\ref{identity-several-neg-1-0}) is the contact orders of two markings (contact orders $k_2$ and $k_1$ for the LHS and contact orders $k_2-k_1$ and $0$ for the RHS) for the rubber invariants. We pushforward the rubber moduli spaces to the moduli space $\overline{M}(D)$ of stable maps to $D$. Since the genus zero double ramification cycle is trivial, it does not depend on the contact orders. Therefore, the LHS of (\ref{identity-several-neg-1-0}) and the RHS of (\ref{identity-several-neg-1-0}) are the same.

\textbf{Induction:}

Now we use the induction to prove the case when $a>0$ and $l>1$. Suppose Identity (\ref{identity-several-neg-1}) is true when $a=N\geq 0$. When $a=N+1$, we apply  the topological recursion relation
\begin{align*}
&\langle [1]_{-k_1},\cdots, [1]_{-k_l}, [\gamma]_{k_{l+1}} \bar{\psi}^{N+1}\rangle_{0,l+1,\beta}^{(X,D)}\\
=&\sum \langle [\gamma]_{k_{l+1}} \bar{\psi}^{N}, \prod_{j\in S_1}[1]_{-k_j},\tilde {T}_{i,k}\rangle_{0,\beta_1,2+|S_1|}^{(X,D)} \langle \tilde {T}_{-i}^k, \prod_{j\in S_2}[1]_{-k_j},[1]_{-k_1},[1]_{-k_2} \rangle_{0,\beta_2,3+|S_2|}^{(X,D)},
\end{align*}
where the sum is over all $\beta_1+\beta_2=\beta$, all indices $i,k$ of basis and $S_1, S_2$ disjoint sets with $S_1\cup S_2=\{3,\ldots,l\}$. The nefness of the divisor $D$ implies that 
\[
\tilde {T}_{-i}^k=[\alpha]_{b} \text{ and } \tilde {T}_{i,k}=[\check{\alpha}]_{-b},
\]
for some positive integer $b\geq k_1$. 

Note that 
\[
\langle [\alpha]_{b}, \prod_{j\in S_2}[1]_{-k_j}, [1]_{-k_1},[1]_{-k_2} \rangle_{0,\beta_2,3+|S_2|}^{(X,D)}=\langle [\alpha]_{b-k_1-k_2},\prod_{j\in S_2}[1]_{-k_j}, [D]_0,[D]_0 \rangle_{0,\beta_2,3+|S_2|}^{(X,D)}
\]
follows from the base case. 

On the  other hand, we have
\[
\langle [\gamma]_{k_{l+1}} \bar{\psi}^{N}, \prod_{j\in S_1}[1]_{-k_j}, [\check{\alpha}]_{-b}\rangle_{0,\beta_1,2+|S_1|}^{(X,D)}=\langle [\gamma]_{k_{l+1}-k_1} \bar{\psi}^{N}, \prod_{j\in S_1}[1]_{-k_j}, [\check{\alpha}]_{-b+k_1}\rangle_{0,\beta_1,2+|S_1|}^{(X,D)}.
\]
This is because, for these invariants, the graph sum in the definition of the relative invariants only has one rubber space and all the markings are in the rubber space. Moreover, the class $\mathfrak c_\Gamma$ does not depend on the value of $k_{l+1}$ and $b$. Therefore, we have the identity.
%....\textcolor{red}{Add more details.}

Therefore, we have
\begin{align*}
    &\sum \langle [\gamma]_{k_{l+1}} \bar{\psi}^{N}, \prod_{j\in S_1}[1]_{-k_j}, \tilde {T}_{i,k}\rangle_{0,\beta_1,2+|S_1|}^{(X,D)} \langle \tilde {T}_{-i}^k, \prod_{j\in S_2}[1]_{-k_j}, [1]_{-k_1},[1]_{-k_2} \rangle_{0,\beta_2,3+|S_2|}^{(X,D)}\\
    =&\sum \langle [\gamma]_{k_{l+1}-k_1} \bar{\psi}^{N}, \prod_{j\in S_1}[1]_{-k_j}, [\check{\alpha}]_{-b+k_1}\rangle_{0,\beta_1,2+|S_1|}^{(X,D)} \langle [\alpha]_{b-k_1},\prod_{j\in S_2}[1]_{-k_j}, [D]_0,[D]_0 \rangle_{0,\beta_2,3+|S_2|}^{(X,D)}\\
    =&\langle [D]_0, [D]_0, \prod_{j\in\{3,\ldots,l\}}[1]_{-k_j},[\gamma]_{k_{l+1}-k_1} \bar{\psi}^{N+1}\rangle_{0,3,\beta}^{(X,D)},\\
\end{align*}
where the third line is the topological recursion relation. We have the identity
\begin{align*}
\langle [1]_{-k_1},\cdots, [1]_{-k_l}, [\gamma]_{k_{l+1}} \bar{\psi}^{N+1}\rangle_{0,l+1,\beta}^{(X,D)}
=\langle [D]_0, [D]_0,\prod_{j\in\{3,\ldots,l\}}[1]_{-k_j}, [\gamma]_{D\cdot \beta-k_{1}} \bar{\psi}^{N+1}\rangle_{0,l+1,\beta}^{(X,D)}.
\end{align*}

Run the above argument multiple times to trade markings with negative contact orders with markings with insertion $[D]_0$. We end up with either one negative contact order or no negative contact order. The former case is the base case II: $l=1$, the latter case is just Identity (\ref{identity-several-neg-1}).

\if{
When $a>0$ and $l>1$, we use induction (on $a$ and $l$) and apply the topological recursion relation for relative Gromov--Witten theory in \cite{FWY}. Suppose the identity (\ref{identity-several-neg-1}) is true for $a=N$. By the topological recursion, the LHS of (\ref{identity-several-neg-1}) when $a=N+1$ is
\[
\sum \langle  [\gamma]_{k_{l+1}} \bar{\psi}^{N}, \prod_{j\in S_1}[1]_{-k_j}, \tilde {T}_{i,k}\rangle_{0,\beta_1,1+|S_1|}^{(X,D)} \langle \tilde {T}_{-i}^k, [1]_{-k_1},[1]_{-k_2},\prod_{j\in S_2}[1]_{-k_j} \rangle_{0,\beta_2,2+|S_2|}^{(X,D)}, 
\]
where the sum is over all $\beta_1+\beta_2=\beta$, all indices $i,k$ of basis, and $S_1, S_2$ disjoint sets with $S_1\cup S_2=\{3,\ldots,l\}$. We first consider
\[
\langle \tilde {T}_{-i}^k, [1]_{-k_1},[1]_{-k_2},\prod_{j\in S_2}[1]_{-k_j} \rangle_{0,\beta_2,2+|S_2|}^{(X,D)}.
\]
The nefness of the divisor $D$ implies that
\[
\tilde {T}_{-i}^k=[\alpha]_{b}
\]
for some positive integer $b$.
}\fi

\end{proof}
%\textcolor{red}{Check the proof again.}

The identity will be slightly different if we add an insertion of $[1]_1$. For the purpose of this paper, we only consider the case when there are no descendant classes. There is also a (more complicated) identity for descendant invariants, but we do not plan to discuss it here. 

%\textcolor{red}{descendant invariants?}
\begin{proposition}\label{prop-several-neg-2}
Given a curve class $\beta$, Let $k_i\in \mathbb Z_{>0}$ for $i\in \{1,\ldots, l+1\}$ such that
\[
D\cdot \beta-1=k_{l+1}-\sum_{i=1}^l k_i\geq 0.
\]
Then we have the following relation.
\begin{align}\label{identity-several-neg-2}
\langle [1]_1, [1]_{-k_1},\cdots, [1]_{-k_l}, [\gamma]_{k_{l+1}} \rangle_{0,l+2,\beta}^{(X,D)}=(D\cdot \beta-1)^l\langle [1]_1, [\gamma]_{D\cdot \beta-1} \rangle_{0,2,\beta}^{(X,D)},
\end{align}
where $\gamma \in H^*(D)$.
\end{proposition}

\begin{proof}
The proof is similar to the proof of Proposition \ref{prop-several-neg-1}. We first consider the LHS of (\ref{identity-several-neg-2}). By definition, every negative contact marking must be in a rubber moduli $\bM^\sim_{\Gamma_i^0}(D)$ labelled by $\Gamma_i^0$ and each rubber moduli $\bM^\sim_{\Gamma_i^0}(D)$ has at least one negative contact marking distributed to it. Similar to the proof of Proposition \ref{prop-several-neg-1}, the nefness of $D$ implies that the last marking (with insertion $[\gamma]_{k_{l+1}}$) has to be distributed to the rubber space.  

Now we examine the first marking. Since the contact order of the first marking is $1$, the nefness of $D$ implies that the first marking and the last marking can not be in different rubber space. On the other hand, if the first marking and the last marking (both with positive contact orders) are in the same rubber space, then we claim the invariant vanishes. This is because, after pushing forward to $\overline{M}_{g,n}(D, \pi_*(\beta))$, there are $(l+1)$-identity class $1$ and the degree of the class $\mathfrak c_\Gamma$ is $l-1$. Applying the string equation $(l+1)$-times implies that the invariant is zero. 
Therefore, there is only one rubber moduli, labelled by $\Gamma^0$, and the first marking can not be distributed to the rubber moduli. The LHS of (\ref{identity-several-neg-2}) is of the following form
\[
\sum_{\mathfrak G\in \mathcal B_\Gamma}\frac{\prod_{e\in E}d_e}{|\Aut(E)|}\sum_\eta \langle [1]_1,\check{\eta}\rangle^{\bullet,\mathfrak c_\Gamma, (X,D)}_{\Gamma^\infty}\langle \eta,[1]_{k_1},\cdots, [1]_{k_l},| \, | [\gamma]_{k_{l+1}} \rangle^{\sim,\mathfrak c_{\Gamma}}_{\Gamma^0}. 
\]

Then the rest of the proof follows from the proof of Proposition \ref{prop-several-neg-1} that the Equation (\ref{identity-several-neg-2}) holds. Note that the contact order of the unique marking in $\eta$ is $D\cdot \beta-1$ instead of $D\cdot \beta$. Therefore, we have the factor $(D\cdot \beta-1)^l$ instead of $(D\cdot \beta)^l$.

\end{proof}

One can add more insertions of $[1]_1$, then we have a similar identity as follows.
\begin{prop}\label{prop-several-neg-3}
Given a curve class $\beta$, Let $k_i\in \mathbb Z_{>0}$ for $i\in \{1,\ldots, l+1\}$ such that
\[
D\cdot \beta-a=k_{l+1}-\sum_{i=1}^l k_i\geq 0.
\]
Then we have the following relation.
\begin{align}\label{identity-several-neg-3}
\langle [1]_1, \ldots, [1]_1, [1]_{-k_1},\cdots, [1]_{-k_l}, [\gamma]_{k_{l+1}} \rangle_{0,a+l+1,\beta}^{(X,D)}=(D\cdot \beta-a)^l\langle [1]_1,\ldots, [1]_1, [\gamma]_{D\cdot \beta-a} \rangle_{0,a+1,\beta}^{(X,D)},
\end{align}
where $\gamma \in H^*(D)$.
\end{prop}
\begin{proof}
The proof is similar to the proof of Proposition \ref{prop-several-neg-1} and Proposition \ref{prop-several-neg-2}. We do not repeat the details.
\end{proof}

\subsection{Degree zero relative invariants}\label{sec-deg-0}

The following invariants will also appear in the $J$-function when plugging in the mirror map:
\[
\langle [1]_1, [1]_{-k_1},\cdots, [1]_{-k_l}, [\gamma]_{k_{l+1}} \rangle_{0,l+2,\beta}^{(X,D)}, \text{ with } D\cdot \beta=0.
\]
We again apply the definition of relative Gromov--Witten invariants with negative contact orders in Section \ref{sec-rel-general} and then pushforward the rubber moduli to $\overline{M}_{0,l+2}(D,\pi_*(\beta))$. Then applying the string equation $(l+1)$-times, then the invariants vanish unless $\beta=0$. 

When $l=1$, the invariants have two markings with positive contact orders and one marking with negative contact order. By direct computation, we have
\[
\langle [1]_1,[1]_{-k_1},[\on{pt}]_{k_{2}}\rangle_{0,3,0}^{(X,D)}=1.    
\]

Therefore, We still need to compute degree zero, genus zero relative invariants with two positive contacts and several negative contacts. By the definition of relative invariants with negative contact orders in Section \ref{sec-rel-general}, the bipartite graphs simplifies to a single vertex of type $0$ and the moduli space is simply the product $\overline{M}_{0,n}\times D$. 

We will have the following result.
\begin{prop}\label{prop-degree-zero}
\begin{align}\label{identity-zero-several-neg}
\langle [1]_1,[1]_{-k_1},\cdots, [1]_{-k_l},[\on{pt}]_{k_{l+1}}\rangle_{0,l+2,0}^{(X,D)}=(-1)^{l-1},    
\end{align}
where $k_1, \ldots, k_l$ are positive integers and
\[
1+k_{l+1}=k_1+\cdots +k_l.
\]

\end{prop}

These invariants can be computed from the $I$-function. The $I$-function is a limit of the $I$-function for the twisted invariants of $B\mathbb Z_r$. The $I$-function is also the restriction of the $I$-function for $(X,D)$ to the degree zero case.

Before computing the corresponding orbifold invariants, we briefly recall the relative-orbifold correspondence of \cite{FWY}: 
\[
r^{m_-}\langle \prod_{i=1}^m \tau_{a_i}(\gamma_i)\rangle_{0,m,\beta}^{X_{D,r}}=\langle \prod_{i=1}^m \tau_{a_i}(\gamma_i)\rangle_{0,m,\beta}^{(X,D)},
\]
where there are $m$ markings in total and $m_-$ of them are markings with negative contact orders; $a_i\in \mathbb Z_{\geq 0}$; $\gamma_i$ are cohomology classes of $X$ (or $D$) when the marking is interior (or relative/orbifold, respectively). We would like to point out that it is important to keep in mind the factor $r^{m_-}$.

We need to consider the $S$-extended $I$-function for the twisted Gromov--Witten invariants of $B\mathbb Z_r$ for sufficiently large $r$, with extended data
\[
S=\{1,-k_1,\ldots,-k_l\}.
\]

The $I$-function is
\begin{align}\label{I-func-BZ}
z\sum \frac{x_1^{a}\prod_{i=1}^l x_{-k_i}^{a_{i}}}{z^{a+\sum_{i=1}^l a_i}a!\prod_{i=1}^l a_i!}\frac{\prod_{b\leq 0, \langle b\rangle=\langle -\frac{1}{r}-\sum_{i=1}^l (a_i(1-\frac{k_i}{r}))\rangle} (bz) }{\prod_{b\leq -\frac{1}{r}-\sum_{i=1}^l (a_i(1-\frac{k_i}{r})), \langle b\rangle=\langle -\frac{1}{r}-\sum_{i=1}^l (a_i(1-\frac{k_i}{r}))\rangle} (bz) } [1]_{\langle a-\sum_{i=1}^l \frac{a_ik_i}{r}\rangle},
\end{align}
where $\langle b \rangle$ is the fractional part of the rational number $b$.

The orbifold mirror map, the $z^0$-coefficient of the $I$-function, is
\[
x_1[1]_{\frac{1}{r}}+\sum_{\{i_1,\ldots, i_j\}\subset \{1,\ldots, l\}} x_{-k_{i_1}}\cdots x_{-k_{i_j}} \prod_{b=0}^{j-1}\left(\frac{ k_{i_1}+\cdots+ k_{i_j}}{r}-b\right)[1]_{-\frac{k_{i_1}+\cdots+ k_{i_j}}{r}}.
\]

Since the expression of the $I$-function and the mirror map looks quite complicated for general $l$. We would like to first start with the computation for the case when $l=2$ for a better explanation of the idea.

\subsubsection{Computation for $l=2$}

When $l=2$, the $I$-function becomes 

\[
z\sum \frac{x_1^{a} x_{-k_1}^{a_{1}}x_{-k_2}^{a_{2}}}{z^{a+a_1+a_2}a! a_1!a_2!}\frac{\prod_{b\leq 0, \langle b\rangle=\langle -\frac{1}{r}-\sum_{i=1}^2 a_i(1-\frac{k_i}{r}))\rangle} (bz) }{\prod_{b\leq -\frac{1}{r}-\sum_{i=1}^2 (a_i(1-\frac{k_i}{r})), \langle b\rangle=\langle -\frac{1}{r}-\sum_{i=1}^2 (a_i(1-\frac{k_i}{r}))\rangle} (bz) } [1]_{\langle a-\sum_{i=1}^l \frac{a_ik_i}{r}\rangle}.
\]
The orbifold mirror map is
\[
x_1[1]_{\frac{1}{r}}+x_{-k_1}[1]_{1-\frac{k_1}{r}}+x_{-k_2}[1]_{1-\frac{k_2}{r}}+x_{-k_1}x_{-k_2}\left(\frac{k_1+k_2}{r}-1\right)[1]_{-\frac{k_1+k_2}{r}}.
\]

We would like to take the coefficient of $x_1x_{-k_1}x_{-k_2}z^{-1}[1]_{-\frac{k_1+k_2-1}{r}}$ of the $I$-function and the $J$-function. The coefficient of the $I$-function is
\[
\frac{k_1+k_2-1}{r}-1.
\]

By the mirror theorem, this coefficient of the $I$-function coincides with the coefficient of the $J$-function:
\begin{align*}
&\langle [1]_{\frac{1}{r}}, [1]_{1-\frac{k_1}{r}}, [1]_{1-\frac{k_2}{r}}, r[1]_{\frac{k_1+k_2-1)}{r}}\rangle_{0,4}^{B \mathbb Z_r, tw}\\
+&\left(\frac{k_1+k_2}{r}-1\right)\langle [1]_{\frac{1}{r}}, [1]_{1-\frac{k_1+k_2}{r}}, r[1]_{\frac{k_1+k_2-1}{r}}\rangle_{0,3}^{B \mathbb Z_r, tw}. 
\end{align*}

Note that the invariant
\[
\langle [1]_{\frac{1}{r}}, [1]_{1-\frac{k_1+k_2}{r}}, r[1]_{\frac{k_1+k_2-1}{r}}\rangle_{0,3}^{B \mathbb Z_r, tw}
\]
coincides with degree zero relative invariant with one negative contact order and the value of the invariants is $1$ by direct computation. Therefore, we have
\[
\frac{k_1+k_2-1}{r}-1=\langle [1]_{\frac{1}{r}}, [1]_{1-\frac{k_1}{r}}, [1]_{1-\frac{k_2}{r}}, r[1]_{\frac{k_1+k_2-1}{r}}\rangle_{0,4}^{B \mathbb Z_r, tw}+\left(\frac{k_1+k_2}{r}-1\right).
\]
Hence, we have
\[
r\langle [1]_{\frac{1}{r}}, [1]_{1-\frac{k_1}{r}}, [1]_{1-\frac{k_2}{r}}, [1]_{\frac{k_1+k_2-1}{r}}\rangle_{0,4}^{B \mathbb Z_r, tw}=-\frac{1}{r}.
\]
We conclude that, the degree zero relative invariant with two negative contact orders is
\begin{align*}
&\langle [1]_1,[1]_{-k_1}, [1]_{-k_2},[\on{pt}]_{k_1+k_2-1}\rangle^{(X,D)}_{0,4,0}\\
=& r^2\langle [1]_{\frac{1}{r}}, [1]_{1-\frac{k_1}{r}}, [1]_{1-\frac{k_2}{r}}, [1]_{\frac{k_1+k_2-1}{r}}\rangle_{0,4}^{B \mathbb Z_r, tw}\\
=&-1.
\end{align*}

\subsubsection{Computation for general $l$}

\begin{proof}[Proof of Proposition \ref{prop-degree-zero}]
We proceed with induction on $l$. Suppose Identity (\ref{identity-zero-several-neg}) is true for $l=N>0$. For $l=N+1$, extracting the coefficient $x_1\prod_{i=1}^{N+1} x_{-k_i}z^{-1}[1]_{-1-\frac{\sum_{i=1}^{N+1} k_i}{r}}$ of the $I$-function (\ref{I-func-BZ}), we have 
\[
\prod_{b=1}^{N} \left( \frac{-1+\sum_{i=1}^{N+1} k_i }{r}-b\right).
\]
The corresponding coefficient of the $J$-function is
\begin{align*}
&\langle [1]_{\frac{1}{r}}, \prod_{i=1}^{N+1} [1]_{1-\frac{k_i}{r}},  r[1]_{\frac{-1+\sum_{i=1}^{N+1}}{r}}\rangle_{0,N+3}^{B \mathbb Z_r, tw}\\
+&\sum_{\{i_1,i_2\}\subset \{1,\ldots,N+1\}}\left(\frac{k_{i_1}+k_{i_2}}{r}-1\right)\langle [1]_{\frac 1 r}, \prod_{i\in\{1,\ldots, N+1\}\setminus \{i_1,i_2\} }[1]_{k_i}, [1]_{1-\frac{k_{i_1}+k_{i_2}}{r}}, r[1]_{\frac{-1+\sum_{i=1}^{N+1} k_i}{r}}\rangle_{0,N+2}^{B \mathbb Z_r, tw}\\
+& \cdots\\
+& \prod_{b=1}^{N} \left( \frac{\sum_{i=1}^{N+1} k_i }{r}-b\right)\langle [1]_{\frac{1}{r}}, [1]_{1-\frac{\sum_{i=1}^{N+1}k_i}{r}},  r[1]_{\frac{-1+\sum_{i=1}^{N+1} k_i}{r}}\rangle_{0,3}^{B \mathbb Z_r, tw}.    
\end{align*}

We further multiply both coefficients of the $I$-function and the $J$-function by $r^{-N}$, then take the constant coefficient (which is the coefficient of the lowest power of $r$). This coefficient of the $I$-function is
\[
\left( -1+\sum_{i=1}^{N+1} k_i \right)^N.
\]
Apply the induction on $l$, the coefficient of the $J$-function is
\begin{align*}
& r^{N+1}\langle [1]_{\frac{1}{r}}, \prod_{i=1}^{N+1} [1]_{1-\frac{k_i}{r}},  [1]_{\frac{-1+\sum_{i=1}^{N+1}}{r}}\rangle_{0,N+3}^{B \mathbb Z_r, tw}\\
+& \sum_{\{i_1,i_2\}\subset \{1,\ldots,N+1\}}\left(k_{i_1}+k_{i_2}\right)(-1)^{N-1}\\
+& \sum_{\{i_1,i_2,i_3\}\subset \{1,\ldots,N+1\}}\left(k_{i_1}+k_{i_2}+k_{i_3}\right)(-1)^{N-2}\\
+& \cdots\\
+& \left( \sum_{i=1}^{N+1} k_i \right)^N.
\end{align*}
The coefficient of the $J$-function can be simplified to
\begin{align*}
& r^{N+1}\langle [1]_{\frac{1}{r}}, \prod_{i=1}^{N+1} [1]_{1-\frac{k_i}{r}},  [1]_{\frac{-1+\sum_{i=1}^{N+1}}{r}}\rangle_{0,N+3}^{B \mathbb Z_r, tw}\\
+& N\left(\sum_{i=1}^N k_i\right)(-1)^{N-1}+ {N\choose 2} \left(\sum_{i=1}^N k_i\right)(-1)^{N-2}+\cdots
+ \left( \sum_{i=1}^{N+1} k_i \right)^N.
\end{align*}
Therefore, the identity of the (coefficients of the) $I$-function and the $J$-function is
\begin{align*}
\left( -1+\sum_{i=1}^{N+1} k_i \right)^N=& r^{N+1}\langle [1]_{\frac{1}{r}}, \prod_{i=1}^{N+1} [1]_{1-\frac{k_i}{r}},  [1]_{\frac{-1+\sum_{i=1}^{N+1}k_i}{r}}\rangle_{0,N+3}^{B \mathbb Z_r, tw}\\
+& N\left(\sum_{i=1}^N k_i\right)(-1)^{N-1}+ {N\choose 2} \left(\sum_{i=1}^N k_i\right)(-1)^{N-2}+\cdots
+ \left( \sum_{i=1}^{N+1} k_i \right)^N.
\end{align*}
The binomial theorem implies that
\[
(-1)^N=r^{N+1}\langle [1]_{\frac{1}{r}}, \prod_{i=1}^{N+1} [1]_{1-\frac{k_i}{r}},  [1]_{\frac{-1+\sum_{i=1}^{N+1}}{r}}\rangle_{0,N+3}^{B \mathbb Z_r, tw}
\]
The orbifold definition of the relative Gromov--Witten invariants with negative contact orders implies Identity (\ref{identity-zero-several-neg}).
\end{proof}

%\textcolor{red}{The above proof needs to be polished significantly!}

%\textcolor{red}{Can we have more than on $[1]_1$? Write down a statement}

\subsection{Relative mirror map}\label{sec-rel-mirror-map}

We recall that the extended relative mirror map (\ref{rel-mirror-map}) is
\begin{align*}
\tau(y,x_1)=\sum_{i=1}^r p_i\log y_i+x_1[1]_{1}+\sum_{\substack{\beta\in \on{NE}(X)\\ D\cdot \beta \geq 2}}\langle [\on{pt}]\psi^{D\cdot \beta-2}\rangle_{0,1,\beta}^Xy^\beta (D\cdot \beta-1)![1]_{-D\cdot \beta}.
\end{align*}

Let
\[
g(y)=\sum_{\substack{\beta\in \on{NE}(X)\\ D\cdot \beta \geq 2}}\langle [\on{pt}]\psi^{D\cdot \beta-2}\rangle_{0,1,\beta}^Xy^\beta (D\cdot \beta-1)!.
\]
Let $\iota: D\hookrightarrow X$ be the inclusion map and $\iota_!:H^*(D)\rightarrow H^*(X)$ be the Gysin pushforward. Recall that
\[
\mathfrak H:=\bigoplus_{i\in \mathbb Z}\mathfrak H_i,
\]
where $\mathfrak H_0=H^*(X)$ and $\mathfrak H_i=H^*(D)$ for $i\neq 0$.
We also denote $\iota_!:\mathfrak H\rightarrow H^*(X)$ for the map such that it is the identity map for the identity sector $\mathfrak H_0$ and is the Gysin pushforward for twisted sectors $\mathfrak H_i$. We first let $x_1=0$, 
then
\[
\iota!J_{(X,D)}(\tau(y),z)=e^{(\sum_{i=1}^r p_i\log y_i+g(y)D)/z}\left(z+\sum_{ \beta\in \on{NE(X)}}\sum_{\alpha}y^{\beta}e^{g(y)(D\cdot \beta)}\left\langle \frac{\phi_\alpha}{z-\bar{\psi}}\right\rangle_{0,1, \beta}^{(X,D)}D\cup \phi^{\alpha}\right).
\]

This has the same effect with the change of variables
\begin{align}\label{relative-mirror-map}
\sum_{i=1}^r p_i\log q_i=\sum_{i=1}^r p_i\log y_i+g(y)D,
\end{align}
or,
\[
q^\beta=e^{g(y)D\cdot \beta}y^\beta.
\]
We will also refer to the change of variables (\ref{relative-mirror-map}) as the \emph{relative mirror map}. In particular, relative mirror map coincides with the local mirror map of $\mathcal O_X(-D)$ after a change of variables $y\mapsto -y$. 
\if{
Then we write the $J$-function $J(\tau(-y),z)$ as
\[
e^{(\sum_{i=1}^r p_i\log q_i)/z}\left(z+\sum_{\substack{(\beta,l)\neq (0,0)\\ \beta\in \on{NE(X)}}}\sum_{\alpha}\frac{(-q)^{\beta}e^{-lg(-y(q))}}{l!}\left\langle \frac{\phi_\alpha}{z-\psi},[1]_1,\ldots, [1]_1\right\rangle_{0,1+l, \beta}^{X}x^l\phi^{\alpha}\right),
\]
where
\[
q^\beta=e^{g(-y)D\cdot \beta}y^\beta.
\]
}\fi

When $x_1\neq 0$, we will be able to compute invariants with more than one positive contact order. We will consider it in the following section.

\section{Theta function computation via relative mirror theorem}

\begin{theorem}\label{thm-main}
Let $X$ be a smooth projective variety with a smooth nef anticanonical divisor $D$. Let $W:=\vartheta_1$ be the mirror proper Landau--Ginzburg potential. Set $q^\beta=t^{D\cdot \beta}x^{D\cdot\beta}$. Then
\[
W=x^{-1}\exp\left(g(y(q))\right),
\]
where
\[
g(y)=\sum_{\substack{\beta\in \on{NE}(X)\\ D\cdot \beta \geq 2}}\langle [\on{pt}]\psi^{D\cdot \beta-2}\rangle_{0,1,\beta}^Xy^\beta (D\cdot \beta-1)!
\]
and $y=y(q)$ is the inverse of the relative mirror map (\ref{relative-mirror-map}).

\end{theorem}

We will prove Theorem \ref{thm-main} in this section through the mirror theorem for the smooth pair $(X,D)$ proved in \cite{FTY}. 

For the purpose of the computation of the theta function:
\[
x\vartheta_1=1+\sum_{n=1}^\infty\sum_{\beta: D\cdot \beta=n+1}n\langle [1]_1,[\on{pt}]_{n}\rangle_{0,\beta,2}^{(X,D)}q^\beta,
\]
we will only consider the $S$-extended $I$-function with
\[
S=\{1\}.
\]
Recall that the $S$-extended $I$-function of $(X,D)$ is defined as follows: 
\[
I_{(X,D)}^{S}(y,x_1,z)=I_++I_-,
\]
where
\begin{align*}
I_+:=&\sum_{\substack{\beta\in \on{NE}(X),k\in \mathbb Z_{\geq 0}\\  k<D\cdot \beta} }J_{X, \beta}(\tau_{0,2},z)y^{\beta}\frac{ x_1^{k}}{z^{k}k!}\frac{\prod_{0<a\leq D\cdot \beta}(D+az)}{D+(D\cdot \beta-k)z}[{1}]_{-D\cdot \beta+k},
\end{align*}
and 
\begin{align*}
I_-:=&\sum_{\substack{\beta\in \on{NE}(X),k\in \mathbb Z_{\geq 0}\\ k\geq D\cdot \beta} }J_{X, \beta}(\tau_{0,2},z)y^{\beta}\frac{ x_1^{k}}{z^{k}k!}\left(\prod_{0<a\leq D\cdot \beta}(D+az)\right)[{ 1}]_{-D\cdot \beta+k}.
\end{align*}

\subsection{Extracting the coefficient of the $J$-function}
We consider the $J$-function 
\[
J(\tau(y,x_1),z),
\]
where
\[
J_{(X,D)}(\tau,z)=z+\tau+\sum_{\substack{(\beta,l)\neq (0,0), (0,1)\\ \beta\in \on{NE(X)}}}\sum_{\alpha}\frac{q^{\beta}}{l!}\left\langle \frac{\phi_\alpha}{z-\bar{\psi}},\tau,\ldots, \tau\right\rangle_{0,1+l, \beta}^{(X,D)}\phi^{\alpha},
\]
and
\begin{align*}
\tau(y,x_1)=\sum_{i=1}^r p_i\log y_i+x_1[1]_{1}+\sum_{\substack{\beta\in \on{NE}(X)\\ D\cdot \beta \geq 2}}\langle [\on{pt}]\psi^{D\cdot \beta-2}\rangle_{0,1,\beta}^Xy^\beta (D\cdot \beta-1)![1]_{-D\cdot \beta}.
\end{align*}
\if{
\[
e^{(\sum_{i=1}^r p_i\log q_i)/z}\left(z+x[1]_1+\sum_{\substack{(\beta,l)\neq (0,0), (0,1)\\ \beta\in \on{NE(X)}}}\sum_{\alpha}\frac{q^{\beta}e^{-lg(y(q))}}{l!}\left\langle \frac{\phi_\alpha}{z-\psi},[1]_1,\ldots, [1]_1\right\rangle_{0,1+l, \beta}^{X}x^l\phi^{\alpha}\right).
\]
}\fi

The sum over the coefficient of $x_1z^{-1}$ of $J(\tau(y,x_1),z)$ that takes value in $[1]_{-n}$, for $n\geq 1$, is the following
\begin{align}\label{coeff-J}
[J(\tau(y,x_1),z)]_{x_1z^{-1}}=& \sum_{\beta: D\cdot \beta\geq 1,n\geq 1}\langle [1]_1,\tau(y),\cdots, \tau(y),[\on{pt}]_{n}\rangle_{0,\beta,k+2}^{(X,D)}q^\beta \\
\notag & +\sum_{\beta: D\cdot \beta=0}\sum_{n\geq 1, k>0}\langle[1]_1,\tau(y),\cdots,\tau(y),[\on{pt}]_n \rangle_{0,\beta,k+2}^{(X,D)}. 
\end{align}

By Proposition \ref{prop-several-neg-2}, we have
\begin{align}\label{identity-deg-geq-1}
& \sum_{\beta: D\cdot \beta \geq 1, n\geq 1}\langle [1]_1,\tau(y),\cdots, \tau(y),[\on{pt}]_{n}\rangle_{0,\beta,k+2}^{(X,D)}q^\beta\\
\notag = &\exp\left(-g(y)\right)\sum_{\beta: D\cdot \beta=n+1,n\geq 1}\langle [1]_1,[\on{pt}]_{n}\rangle_{0,\beta,2}^{(X,D)}q^\beta,
\end{align}
where
\[
g(y)=\sum_{\substack{\beta\in \on{NE}(X)\\ D\cdot \beta \geq 2}}\langle [\on{pt}]\psi^{D\cdot \beta-2}\rangle_{0,1,\beta}^Xy^\beta (D\cdot \beta-1)!
\]
and
\[
q^\beta=e^{g(y)D\cdot \beta}y^\beta.
\]

When $D\cdot \beta=0$, the invariants are studied in Section \ref{sec-deg-0}. As mentioned at the beginning of Section \ref{sec-deg-0}, we need to have $\beta=0$. The degree zero invariants are computed in Proposition \ref{prop-degree-zero}. Therefore, we have
\begin{align}\label{identity-deg-0}
&\sum_{\beta: D\cdot \beta=0}\sum_{n\geq 1, k>0}\langle[1]_1,\tau(y),\cdots,\tau(y),[\on{pt}]_n \rangle_{0,\beta,k+2}^{(X,D)}\\
\notag =& g(y)+\sum_{l\geq 2} g(y)^l(-1)^{l-1}\\
\notag =& -\exp\left(-g(y)\right)+1.
\end{align}

Therefore, (\ref{identity-deg-geq-1}) and (\ref{identity-deg-0}) imply that (\ref{coeff-J}) is
\begin{align}\label{coeff-J-1}
& [J(\tau(y,x_1),z)]_{x_1z^{-1}}\\
\notag =& \exp\left(-g(y)\right)\sum_{\beta: D\cdot \beta=n+1,n\geq 1}\langle [1]_1,[\on{pt}]_{n}\rangle_{0,\beta,2}^{(X,D)}q^\beta  -\exp\left(-g(y)\right)+1. 
\end{align}

Note that (\ref{coeff-J-1}) is not exactly the generating function of relative invariants in the theta function $\vartheta_1$. We want to compute $\sum_{\beta: D\cdot \beta=n+1,n\geq 1}n\langle [1]_1,[\on{pt}]_{n}\rangle_{0,\beta,2}^{(X,D)}q^\beta$ instead of $\sum_{\beta: D\cdot \beta=n+1,n\geq 1}\langle [1]_1,[\on{pt}]_{n}\rangle_{0,\beta,2}^{(X,D)}q^\beta$. 

Write 
\[
D=\sum_{i=1}^r m_ip_i
\]
for some $m_i\in \mathbb Z_{\geq 0}$.
In order to compute $\vartheta_1$, we apply the operator $\Delta_D=\sum_{i=1}^r m_iy_i\frac{\partial}{\partial y_i}-1$ to the $J$-function $J(\tau(y,x_1),z)$. Then (\ref{coeff-J-1}) becomes
\begin{align}\label{coeff-J-der}
    & \left(-\sum_{i=1}^r m_iy_i\frac{\partial (g(y))}{\partial y_i}\right)\exp\left(-g(y)\right)\sum_{\beta: D\cdot \beta=n+1}\langle [1]_1,[\on{pt}]_{n}\rangle_{0,\beta,2}^{(X,D)}q^\beta \\
\notag    +& \exp\left(-g(y)\right)\sum_{i=1}^r m_iy_i\sum_{j=1}^r\frac{\partial q_j}{\partial y_i}\frac{\partial}{\partial q_j}\sum_{\beta: D\cdot \beta=n+1}\langle [1]_1,[\on{pt}]_{n}\rangle_{0,\beta,2}^{(X,D)}q^\beta \\
\notag    -& \exp\left(-g(y)\right)\sum_{\beta: D\cdot \beta=n+1}\langle [1]_1,[\on{pt}]_{n}\rangle_{0,\beta,2}^{(X,D)}q^\beta+ \left(1+\sum_{i=1}^r m_iy_i\frac{\partial (g(-y))}{\partial y_i}\right)\exp(-g(y))-1.
    \end{align}

We compute the partial derivatives
\[
\frac{\partial q_j}{\partial y_i}=\left\{
\begin{array}{cc}
    y_je^{m_jg(y)}m_j\frac{\partial g(y)}{\partial y_i} & j\neq i; \\
    e^{m_j g(y)}+y_je^{m_jg(y)}m_j\frac{\partial g(y)}{\partial y_j}  & j=i.
\end{array}
\right.
\]
Therefore,
\begin{align*}
&\sum_{i=1}^r m_iy_i\sum_{j=1}^r\frac{\partial q_j}{\partial y_i}\frac{\partial}{\partial q_j}\\
=&\sum_{j=1}^r \left(m_jy_j\left(e^{m_j g(y)}+y_je^{m_jg(y)}m_j\frac{\partial g(y)}{\partial y_j}\right)+m_iy_i\sum_{j\neq i} y_je^{m_jg(y)}m_j\frac{\partial g(y)}{\partial y_i} \right)\frac{\partial}{\partial q_j}\\
=&\sum_{j=1}^r\left(1+\sum_{i=1}^r m_iy_i\frac{\partial g(y)}{\partial y_i}\right)m_jq_j\frac{\partial}{\partial q_j}\\
=&\left(1+\sum_{i=1}^r m_iy_i\frac{\partial g(y)}{\partial y_i}\right)\sum_{j=1}^rm_jq_j\frac{\partial}{\partial q_j}.
\end{align*}

Hence, (\ref{coeff-J-der}) is

    \begin{align*}
    &\left(-\sum_{i=1}^r m_iy_i\frac{\partial (g(y))}{\partial y_i}\right)\exp\left(-g(y)\right)\sum_{\beta: D\cdot \beta=n+1}\langle [1]_1,[\on{pt}]_{n}\rangle_{0,\beta,2}^{(X,D)}q^\beta \\
    & +\exp\left(-g(y)\right)\left(1+\sum_{i=1}^r m_iy_i\frac{\partial g(y)}{\partial y_i}\right)\sum_{j=1}^rm_jq_j\frac{\partial}{\partial q_j}\sum_{\beta: D\cdot \beta=n+1}\langle [1]_1,[\on{pt}]_{n}\rangle_{0,\beta,2}^{(X,D)}q^\beta \\
    & -\exp\left(-g(y)\right)\sum_{\beta: D\cdot \beta=n+1}\langle [1]_1,[\on{pt}]_{n}\rangle_{0,\beta,2}^{(X,D)}q^\beta + \left(1+\sum_{i=1}^r m_iy_i\frac{\partial (g(-y))}{\partial y_i}\right)\exp(-g(y))-1\\
    = &\left(-1-\sum_{i=1}^r m_iy_i\frac{\partial (g(y))}{\partial y_i}\right)\exp\left(-g(y)\right)\left(\sum_{\beta: D\cdot \beta=n+1}\langle [1]_1,[\on{pt}]_{n}\rangle_{0,\beta,2}^{(X,D)}q^\beta \right)\\
    & +\exp\left(-g(y)\right)\left(1+\sum_{i=1}^r m_iy_i\frac{\partial g(y)}{\partial y_i}\right)\left(\sum_{j=1}^r\sum_{\beta: D\cdot \beta=n+1}(n+1)\langle [1]_1,[\on{pt}]_{n}\rangle_{0,\beta,2}^{(X,D)}q^\beta +1\right)-1\\
    =& \left(1+\sum_{i=1}^r m_iy_i\frac{\partial (g(y))}{\partial y_i}\right)\exp\left(-g(y)\right)\left(\sum_{\beta: D\cdot \beta=n+1}n\langle [1]_1,[\on{pt}]_{n}\rangle_{0,\beta,2}^{(X,D)}q^\beta +1\right)-1.
\end{align*}

\subsection{Extracting the coefficient of the $I$-function}

Recall that, when $\beta=0$, we have
\[
J_{X, 0}(\tau_{0,2},z)=z.
\]
When $\beta\neq 0$, we have
\[
J_{X, \beta}(\tau_{0,2},z)=e^{\tau_{0,2}/z}\sum_\alpha\left\langle\psi^{m-2}\phi_\alpha\right\rangle_{0,1,\beta}^Xy^\beta\phi^\alpha\left(\frac{1}{z}\right)^{m-1}
\]
and 
\[
m=\dim_{\mathbb C} X+D\cdot \beta-\deg (\phi_\alpha)\geq D\cdot\beta.
\]

The $I$-function can be expanded as
\[
I=z+x_1[1]_1 +\tau_{0,2}+\sum_{\substack{\beta\in \on{NE}(X)\\ D\cdot \beta \geq 2}}\langle [\on{pt}]\psi^{D\cdot \beta-2}\rangle_{0,1,\beta}^Xy^\beta (D\cdot \beta-1)![1]_{-D\cdot \beta}+\sum_{k=1}^{\infty}I_{-k}z^{-k}.
\]

We would like to sum over the coefficient of $x_1z^{-1}$ of the $I$-function that takes value in $[1]_{-n}$ for $n\geq 1$. By direct computation, the sum of the coefficients is 
\begin{align}\label{coeff-I}
[I(y,z)]_{x_1z^{-1}}=\sum_{\beta: D\cdot \beta =n+1, n\geq 1}\langle [\on{pt}]\psi^{n-1}\rangle_{0,1,\beta}^Xy^\beta \frac{(n+1)!}{n}.
\end{align}

We apply the operator 
\[
\Delta_D=\sum_{i=1}^r m_iy_i\frac{\partial}{\partial y_i}-1
\]
to the $I$-function $I(y,z)$. Then (\ref{coeff-I}) becomes 
\[
\sum_{\beta: D\cdot \beta =n+1, n\geq 1}\langle [\on{pt}]\psi^{n-1}\rangle_{0,1,\beta}^X(y)^\beta (n+1)!.
\]

\subsection{Matching}
The relative mirror theorem of \cite{FTY} states that the coefficients of the $I$-function and the $J$-function are the same. Therefore, we have
\begin{align}\label{identity-coeff-I-J}
& \sum_{\beta: D\cdot \beta =n+1, n\geq 1}\langle [\on{pt}]\psi^{n-1}\rangle_{0,1,\beta}^X(y)^\beta (n+1)!\\
\notag = & \left(1+\sum_{i=1}^r m_iy_i\frac{\partial (g(y))}{\partial y_i}\right)\exp\left(-g(y)\right)\left(\sum_{\beta: D\cdot \beta=n+1}n\langle [1]_1,[\on{pt}]_{n}\rangle_{0,\beta,2}^{(X,D)}q^\beta +1\right)-1.
\end{align}
Recall that
\[
g(y)=\sum_{\substack{\beta\in \on{NE}(X)\\ D\cdot \beta \geq 2}}\langle [\on{pt}]\psi^{D\cdot \beta-2}\rangle_{0,1,\beta}^Xy^\beta (D\cdot \beta-1)!.
\]
Therefore 
\begin{align*}
&\sum_{i=1}^r m_iy_i\frac{\partial (g(y))}{\partial y_i}\\
=&\sum_{\substack{\beta\in \on{NE}(X)\\ D\cdot \beta \geq 2}}\langle [\on{pt}]\psi^{D\cdot \beta-2}\rangle_{0,1,\beta}^Xy^\beta (D\cdot \beta)!\\
=& \sum_{\beta: D\cdot \beta =n+1, n\geq 1}\langle [\on{pt}]\psi^{n-1}\rangle_{0,1,\beta}^X(y)^\beta (n+1)!
\end{align*}

we conclude that (\ref{identity-coeff-I-J}) becomes 
\begin{align*}
& 1+ \sum_{\beta: D\cdot \beta =n+1, n\geq 1}\langle [\on{pt}]\psi^{n-1}\rangle_{0,1,\beta}^X(y)^\beta (n+1)!\\
\notag = & \left(1+\sum_{\beta: D\cdot \beta =n+1, n\geq 1}\langle [\on{pt}]\psi^{n-1}\rangle_{0,1,\beta}^X(y)^\beta (n+1)!\right)\exp\left(-g(y)\right)\left(\sum_{\beta: D\cdot \beta=n+1}n\langle [1]_1,[\on{pt}]_{n}\rangle_{0,\beta,2}^{(X,D)}q^\beta +1\right).
\end{align*}
Therefore
\[
1=\exp\left(-g(y)\right)\left(\sum_{\beta: D\cdot \beta=n+1}n\langle [1]_1,[\on{pt}]_{n}\rangle_{0,\beta,2}^{(X,D)}q^\beta +1\right).
\]
We have
\begin{align}
    1+\sum_{\beta: D\cdot \beta=n+1}n\langle [1]_1,[\on{pt}]_{n}\rangle_{0,\beta,2}^{(X,D)}q^\beta=\exp\left(g(y(q))\right),
\end{align}
where $y=y(q)$ is the inverse mirror map. This concludes Theorem \ref{thm-main}.

\section{Toric varieties and the open mirror map}

In this section, We will specialize our result to the toric case. The proper Landau--Ginzburg potential can be computed explicitly.

Let $X$ be a toric variety with a smooth, nef anticanonical divisor $D$. 
Recall that the small $J$-function for absolute Gromov--Witten theory of $X$ is
\[
J_{X}(z)=e^{\sum_{i=1}^r p_i\log q_i/z}\left(z+\sum_{\substack{(\beta,l)\neq (0,0), (0,1)\\ \beta\in \on{NE(X)}}}\sum_{\alpha}\frac{q^{\beta}}{l!}\left\langle \frac{\phi_\alpha}{z-\psi}\right\rangle_{0,1, \beta}^{X}\phi^{\alpha}\right),
\]
where $\tau_{0,2}=\sum_{i=1}^r p_i \log q_i\in H^2(X)$; $\{\phi_\alpha\}$ is a basis of $H^*(X)$; $\{\phi^\alpha\}$ is the dual basis under the Poincar\'e pairing.

By \cite{Givental98}, the $I$-function for a toric variety $X$ is

\[
I_{X}(y,z)
= ze^{t/z}\sum_{\beta\in \on{NE}(X)}y^{\beta}\left(\prod_{i=1}^{m}\frac{\prod_{a\leq 0}(D_i+az)}{\prod_{a \leq D_i\cdot \beta}(D_i+az)}\right),
\]

where $t=\sum_{a=1}^r p_a \log y_a$, $y^\beta=y_1^{p_1\cdot \beta}\cdots y_r^{p_r\cdot \beta}$ and $D_i$'s are toric divisors.

The $I$-function can be expanded as
\[
z+\sum_{a=1}^r p_a \log y_a+\sum_{j}D_j\sum_{\substack{c_1(X)\cdot \beta=0,D_j\cdot \beta<0\\ D_i\cdot \beta\geq 0, \forall i\neq j}} \frac{(D_j\cdot \beta-1)!}{\prod_{i=1, i\neq j}^m (D_i\cdot \beta)!}y^\beta+O(z^{-1}).
\]

%\textcolor{red}{Add explanation.......}

The $J$-function and the $I$-function are related by the following change of variables, called the (absolute) mirror map:
\[
\sum_{i=1}^r p_i\log q_i=\sum_{i=1}^r p_i\log y_i+\sum_{j}D_j\sum_{\substack{c_1(X)\cdot \beta=0,D_j\cdot \beta<0\\ D_i\cdot \beta\geq 0, \forall i\neq j}} \frac{(D_j\cdot \beta-1)!}{\prod_{i=1, i\neq j}^m (D_i\cdot \beta)!}y^\beta
\]

Set $z=1$, the coefficient of the $1\in H^0(X)$ of the $J$-function is
\[
\sum_{\substack{\beta\in \on{NE}(X)\\ D\cdot \beta \geq 2}}\langle [\on{pt}]\psi^{D\cdot \beta-2}\rangle_{0,1,\beta}^X
\]
The corresponding coefficient of the $I$-function is
\begin{align*}
&\sum_{D_i\cdot \beta\geq 0, \forall i} \frac{1}{\prod_{i=1}^m (D_i\cdot \beta)!}y^\beta\\
=&\sum_{\substack{\beta\in \on{NE}(X)\\ D\cdot \beta \geq 2}} \frac{1}{\prod_{i=1}^m (D_i\cdot \beta)!}y^\beta.
\end{align*}

\subsection{Toric Fano varieties}

When $X$ is Fano, the absolute mirror map is trivial. Then we have
\[
g(y)=\sum_{D_i\cdot \beta\geq 0, \forall i} \frac{ (D\cdot \beta-1)!}{\prod_{i=1}^m (D_i\cdot \beta)!}y^\beta
.
\]

Then Theorem \ref{thm-main} specialize to
\[
W=\exp\left(\sum_{D_i\cdot \beta\geq 0, \forall i} \frac{ (D\cdot \beta-1)!}{\prod_{i=1}^m (D_i\cdot \beta)!}y(q)^\beta
 \right),
\]
where $y(q)$ is the inverse to the relative mirror map
\begin{align*}
\sum_{i=1}^r p_i\log q_i=\sum_{i=1}^r p_i\log y_i+g(y)D.
\end{align*}

If we further specialize the result to dimension $2$ case, we recover the main result of \cite{GRZ}.

\subsection{Toric varieties with a smooth, nef anticanonical divisor}\label{sec-toric-semi-Fano}

Theorem \ref{thm-main} specializes to
\[
W=\exp\left(\sum_{D_i\cdot \beta\geq 0, \forall i} \frac{ (D\cdot \beta-1)!}{\prod_{i=1}^m (D_i\cdot \beta)!}y(q)^\beta
 \right),
\]
where $y(q)$ is the inverse to the relative mirror map
\begin{align*}
\sum_{i=1}^r p_i\log q_i=\sum_{i=1}^r p_i\log y_i+g(y)D + \sum_{j}D_j\sum_{\substack{c_1(X)\cdot \beta=0,D_j\cdot \beta<0\\ D_i\cdot \beta\geq 0, \forall i\neq j}} \frac{(-D_j\cdot \beta-1)!}{\prod_{i=1, i\neq j}^m (D_i\cdot \beta)!}y^\beta.
\end{align*}

%\textcolor{red}{Is it a sign error?}

Let $X$ be a Fano variety with a smooth anticanonical divisor $D$. In \cite{GRZ}, the authors proposed that the mirror proper Landau--Ginzburg potential is the open mirror map of the local Calabi--Yau $\mathcal O_X(-D)$. When $X$ is toric, the open mirror map for the toric Calabi--Yau $\mathcal O_X(-D)$ has been computed in \cite{CCLT} and \cite{CLT} (see also \cite{You20} for the computation in terms of relative Gromov--Witten invariants). 

The SYZ mirror construction for a toric Calabi--Yau manifold $Y$ was constructed in \cite{CCLT} and \cite{CLT}. We specialize to the case when $Y=\mathcal O_X(-D)$ where $D$ is a smooth, nef, anticanonical divisor of $X$. Note that, we do not need to assume $X$ is Fano.
The SYZ mirror of $Y$ is modified by the instanton corrections. Following \cite{Auroux07} and \cite{Auroux09}, the SYZ mirror of a toric Calabi--Yau manifold was constructed in \cite{CCLT} and \cite{CLT} where the instanton corrections are given by genus zero open Gromov--Witten invariants. These open Gromov--Witten invariants are virtual counts of holomorphic disks in $Y$ bounded by fibers of the Gross fibration. It was shown in \cite{CCLT} and \cite{CLT} that the generating function of these open invariants is the inverse of the mirror map for $Y$. This generating function of these open invariants is referred to as the open mirror map in \cite{GRZ}.

Comparing the open mirror map of \cite{CCLT} and \cite{CLT} with our relative mirror map, we directly have 
\begin{theorem}\label{thm-toric-open}
Let $(X,D)$ be a smooth log Calabi--Yau pair, such that $X$ is toric and $D$ is nef. The proper Landau--Ginzburg potential of $(X,D)$ is the open mirror map of the local Calabi--Yau manifold $\mathcal O_X(-D)$.
\end{theorem}

\subsection{Beyond the toric case}

Beyond the toric setting, the conjecture of \cite{GRZ} is also expected to be true as long as we assume the following principal (open-closed duality) in mirror symmetry.

\begin{conjecture}\label{conj-open-closed}
The instanton corrections of a local Calabi--Yau manifold $\mathcal O_X(-D)$ is the inverse mirror map of the local mirror theorem that relates local Gromov--Witten invariants with periods. 
\end{conjecture}

Open Gromov--Witten invariants have not been defined in the general setting. Moreover, open Gromov--Witten invariants are more difficult to compute. On the other hand, the local mirror map can usually be computed. As we have seen in Section \ref{sec-rel-mirror-map}, the local mirror map and the relative mirror map coincide. Therefore, we have the following result.

\begin{corollary}
Assuming Conjecture \ref{conj-open-closed}, the proper Landau--Ginzburg potential is the open mirror map.
\end{corollary}

\section{Fano varieties and quantum periods}

For a Fano variety $X$, the function $g(y)$ is closely related to the quantum period of $X$. In fact, we have
\begin{theorem}\label{thm-quantum-period}
The function $g(y)$ coincides with the anti-derivative of the regularized quantum period.
\end{theorem}
\begin{proof}
We recall that 
\[
g(y)=\sum_{\substack{\beta\in \on{NE}(X)\\ D\cdot \beta \geq 2}}\langle [\on{pt}]\psi^{D\cdot \beta-2}\rangle_{0,1,\beta}^Xy^\beta (D\cdot \beta-1)!.
\]
We consider the change of variables
\[
y^\beta=t^{D\cdot \beta}.
\]
Then the derivative $\frac{d}{dt}$ of $g(t)$ is 
\[
\sum_{\substack{\beta\in \on{NE}(X)\\ D\cdot \beta \geq 2}}\langle [\on{pt}]\psi^{D\cdot \beta-2}\rangle_{0,1,\beta}^Xt^{D\cdot\beta} (D\cdot \beta)!,
\]
which is precisely the regularized quantum period in \cite{CCGGK}.
\end{proof}

Following the Fanosearch program, the regularized quantum period of a Fano variety coincides with the classical period of its mirror Laurent polynomial. A version of the relation between the quantum period and the classical period was obtained in \cite{TY20b} using the formal orbifold invariants of infinite root stacks \cite{TY20c}. Combining with the Fanosearch program, one can explicitly compute the proper Landau--Ginzburg potential of a Fano variety as long as one knows its mirror Lanurent polynomial. In particular, we have found explicit expressions of the proper Landau--Ginzburg potentials for all Fano threefolds using the expression of the quantum periods in \cite{CCGK}.

\begin{example}
We consider the Fano threefold $V_{10}$ in \cite{CCGK}*{Section 12}. It is a Fano threefold with Picard rank 1, Fano index 1, and degree $10$. It can be considered as a complete intersection in the Grassmannian $\on{Gr}(2,5)$. Following \cite{CCGK}, the quantum period is
\[
G_{V_{10}}(y)=e^{-6y}\sum_{l=0}^\infty\sum_{m=0}^\infty (-1)^{l+m}y^{l+m}\frac{((l+m)!)^2(2l+2m)!}{(l!)^5(m!)^5}(1-5(m-l)H_m),
\]
where $H_m$ is the $m$-th harmonic number.

Therefore,
\[
g_{V_{10}}(y)=e^{-6y}\sum_{l=0}^\infty\sum_{m=0}^\infty (-1)^{l+m}y^{l+m}\frac{((l+m)!)^2(2l+2m)!}{(l!)^5(m!)^5}(1-5(m-l)H_m)(l+m-1)!
\]

The proper Landau--Ginzburg potential is
\[
W=x^{-1}\exp\left(g_{V_{10}}(y(tx))\right),
\]
where $y(tx)$ is the inverse of 
\[
tx=y\exp\left(g_{V_{10}}(y)\right).
\]
\end{example}

Similarly, one can compute the proper Landau--Ginzburg potential for all Fano threefold using the quantum period in \cite{CCGK}. Moreover, there are large databases \cite{CK22} of quantum periods for Fano manifolds which can be used to compute the proper Landau--Ginzburg potential.

%\textcolor{red}{check variables}

%\begin{example}
%Consider a fourfold example.....
%\end{example}

%\textcolor{red}{add examples.}

\begin{remark}
We noticed that H. Ruddat \cite{Ruddat} has been working on the relation between the proper Landau--Ginzburg potential and the classical period. This can also be seen from Theorem \ref{thm-quantum-period} because it is expected from mirror symmetry that the regularized quantum period of a Fano variety equals the classical period of the mirror Laurent polynomial. The Laurent polynomials are considered as the potential for the weak, non-proper, Landau--Ginzburg model of \cite{Prz07}, \cite{Prz13}. Therefore, Theorem \ref{thm-quantum-period} provides an explicit relation between the proper and non-proper Landau--Ginzburg potentials.
\end{remark}

\bibliographystyle{amsxport}
\bibliography{main}

\end{document}